\documentclass[a4paper,12pt]{article}
\usepackage{pstricks}
\usepackage{pstricks-add}

\usepackage{amsmath, amsfonts, amssymb, amstext, amsthm} 
\usepackage{graphicx, subfigure} 
\usepackage{graphicx, subfigure} 
\usepackage[affil-it]{authblk} 
\usepackage{tikz}
\usepackage[titletoc,title]{appendix}

\def\R{\mathbb{R}}

\def\C{\mathbb{C}}
\def\e{\varepsilon}

\def\a{\alpha}
\def\b{\beta}
\def\la{\lambda}
\def\d{\delta}
\def\k{\kappa}
\def\o{\omega}

\def\nub{\overline{\nu}}
\def\mub{\overline{\mu}}

\def\fp{f'(0)}

\def\vp{\varphi}
\def\Gt2{\widetilde{\Gamma_2}}
\def\su{\underline{u}}
\def\sv{\underline{v}}

\def\cs{\overline{c}}
\def\las{\overline{\lambda}}

\def\phib{\overline{\phi}}
\def\phit{\tilde{\phi}}

\def\uh{\hat{u}}
\def\vh{\hat{v}}
\def\Uh{\hat{U}}
\def\Vh{\hat{V}}

\def\ut{\tilde{u}}
\def\vt{\tilde{v}}

\def\xd{\xi^2}
\renewcommand{\Re}[1]{{\cal R}e  \left ( #1 \right )}
\renewcommand{\Im}[1]{ {\cal I}m  \left ( #1 \right )}
\DeclareMathOperator{\sgn}{sgn}
\def\lp{\left(}
\def\rp{\right)}
\def\lb{\left|}
\def\rb{\right|}
\def\lV{\left\Vert}
\def\rV{\right\Vert}

\def\MD{\mathcal{D}}
\def\MC{\mathcal{C}}

\newtheorem{theorem}{Theorem}
\newtheorem{prop}[theorem]{Proposition}
\newtheorem{lem}[theorem]{Lemma}
\newtheorem{cor}[theorem]{Corollary}
\newtheorem{lemA}{Lemma}[section]

\theoremstyle{definition}
\newtheorem{definition}{Definition}[section]

\theoremstyle{remark}
\newtheorem{remark}{Remark}[section]

\title{Uniform dynamics for Fisher-KPP
propagation driven by a line of fast diffusion
under a singular limit}
  \author{Antoine Pauthier%
  \thanks{e-mail: \texttt{antoine.pauthier@math.univ-toulouse.fr}}}
\affil{Institut de Math\'ematiques de Toulouse ; UMR5219 \\ Universit\'e de Toulouse ; CNRS \\ UPS IMT, F-31062 Toulouse Cedex 9, France}

\begin{document}
	\maketitle

\begin{abstract}
The purpose of this paper is to understand the links between a model introduced in 2012 by H. Berestycki, J.-M. Roquejoffre
and L. Rossi and a nonlocal model studied by the author in 2014. The general question is to investigate the influence of a line of fast 
diffusion on Fisher-KPP propagation. In the initial model, the exchanges are modeled by a Robin boundary condition, whereas in the nonlocal model the 
exchanges are described by integral terms. 
For both models was showed the existence of an enhanced spreading in the direction of the line.
One way to retrieve the local model from the nonlocal one is to consider 
integral terms tending to Dirac masses. The question is then how the dynamics given by the nonlocal
model resembles the local one. We show here that the nonlocal dynamics tends to the local one in a rather strong sense.
\end{abstract}

\section{Introduction}
\paragraph{Presentation of the models}
This paper is concerned with the large time behaviour and propagation phenomena 
for reaction-diffusion equations with a line of fast diffusion. Our model will degenerate, when a small parameter tends to 0,
to a singular limit. The results that we will present will be uniform with respect to this small parameter.
The model under study (\ref{RPeps}) was introduced in 2014 by the author in \cite{Pauthier}.
\begin{equation}
\label{RPeps}
\begin{cases}
\partial_t u-D \partial_{xx} u = -\mub u+\int \nu_\e(y)v(t,x,y)dy & x \in \R,\ t>0 \\
\partial_t v-d\Delta v = f(v) +\mu_\e(y)u(t,x)-\nu_\e(y)v(t,x,y) & (x,y)\in \R^2,\ t>0.
\end{cases}
\end{equation}
A two-dimensional environment (the plane $\R^2$) includes a line (the line $\{(x,0), x\in\R\}$) in which 
fast diffusion takes place while reproduction and usual diffusion only occur outside the line. For the sake of simplicity, we will refer 
to the plane as ``the field`` and the line as ``the road``, as a reference to the biological situations. 
The density of the population is designated by $v=v(t,x,y)$ in the field, 
and $u=u(t,x)$ on the road. Exchanges of population between the road and field are defined by two nonnegative compactly supported functions
$\nu$ and $\mu.$ These functions will be called the \textit{exchange functions}. The density of individuals who jump from a 
point of the field to the road is represented by $y\mapsto \nu_\e(y)$, from the road to a point of the field by 
$y \mapsto \mu_\e(y),$ with the following scaling with $\e>0:$
$$
\nu_\e(y)=  \frac{1}{\e}\nu\lp\frac{y}{\e}\rp,\ \mu_\e(y)=  \frac{1}{\e}\mu\lp\frac{y}{\e}\rp.
$$
We use the notation $\mub=\int \mu,$ $\nub=\int\nu.$ 

It is easy to see that
$\nu_\e\to \nub\d$ and $\mu_\e\to\mub\d$ as $\e\to 0$ in the distribution sense, where $\d=\d_0$, the Dirac function in 0.
Hence, at least formally, the above system (\ref{RPeps}) tends to the following system (\ref{BRReq2})  
where exchanges of population are localised on the road:
\begin{equation}
\label{BRReq2}
\begin{cases}
\partial_t u-D \partial_{xx} u= \nub v(t,x,0)-\mub u & x\in\R,\ t>0\\
\partial_t v-d\Delta v=f(v) & (x,y)\in\R\times\R^*,\ t>0\\
v(t,x,0^+)=v(t,x,0^-), & x\in\R,\ t>0 \\
-d\left\{ \partial_y v(t,x,0^+)-\partial_y v(t,x,0^-) \right\}=\mub u(t,x)-\nub v(t,x,0) & x\in\R,\ t>0.
\end{cases}
\end{equation}
This model was introduced in 2013 in \cite{BRR1} by H. Berestycki, J.-M. Roquejffre and L. Rossi to describe 
biological invasions in a plane when a strong diffusion takes place on a line.
Considering a nonnegative, compactly supported initial datum
$(u_0,v_0)\neq(0,0)$, the authors proved  
the existence of an  asymptotic speed of spreading $c^*_0$ in
the direction of the road for the system (\ref{BRReq2}) for a KPP-type nonlinearity. 
They also explained the dependence of $c^*_0$ on $D,$ the coefficient of diffusion on the road.

The same kind of results was investigated in \cite{Pauthier} for our system (\ref{RPeps}) with fixed $\e,$ 
say $\e=1$ for instance. The main theorem was the following spreading result:
 \begin{theorem}\label{spreadingthm}
 We consider the nonlocal system (\ref{RPeps}) with a KPP-type nonlinearity and two nonnegative compactly supported exchange functions 
 $\nu$ and $\mu,$ with $\nub,\mub>0.$
Let $(u_\e,v_\e)$ be a solution of (\ref{RPeps}) with a nonnegative, compactly supported initial datum $(u_0,v_0)$.
Then, for all $\e>0,$ there exists an asymptotic speed of spreading $c^*_\e$ and a unique positive bounded stationary solution of (\ref{RPeps})
$(U_\e,V_\e)$ such that, 
pointwise in $y$, we have: 
 \begin{itemize}
  \item for all $c>c^*_\e$, $\displaystyle\lim_{t\to\infty}\sup_{|x|\geq ct}(u(x,t),v(x,y,t)) = (0,0)$ ;
  \item for all $c<c^*_\e$, $\displaystyle\lim_{t\to\infty}\inf_{|x|\leq ct}(u(x,t),v(x,y,t)) = (U_\e,V_\e)$.
 \end{itemize}
\end{theorem}
This result is similar to the one showed in \cite{BRR1}, where 
the steady state is given by $(U_0,V_0)=\lp \frac{\nub}{\mub},1\rp.$
So, a natural question is: are the limits in Theorem \ref{spreadingthm} uniform in $\e$ ?
A first reasonable guess is that the spreading speed $c^*_\e$ tends to the spreading speed $c^*_0$
associated to the limit model (\ref{BRReq2}). This point will be developed in the second section of this paper.
The problem of the commutation of the two limits $\e\to0$ and $t\to+\infty$ in Theorem \ref{spreadingthm} is more intricate.
It is the main purpose of this paper. It involves both tools
from functional analysis concerning singular limits and reaction-diffusion methods about spreading phenomena. 

\paragraph{Assumptions}
We always assume that the initial datum $u_0,v_0$ is nonnegative, continuous and compactly supported, with 
$(u_0,v_0)\not\equiv(0,0).$ The reaction term $f$ satisfies:
\begin{equation}\label{kppassumtions}
 f\in C^2([0,1]), \ f(0)=f(1)=0, \ \forall s\in (0,1),\ 0<f(s)\leq f'(0)s.
\end{equation}
Such a reaction term is, as usual, referred to as of KPP type (from the article of Kolmogorov, Petrovsky and Piskounov \cite{KPP}).
We extend it to a uniformly Lipschitz function outside $(0,1),$ satisfying
\begin{equation}\label{lipf}
 \lim_{s\to +\infty} \frac{f(s)}{s}<-2\frac{\nub^2}{d}.
\end{equation}
In particular, we suppose $2\frac{\nub^2}{d}<\lV f\rV_{Lip}<+\infty.$ The assumption (\ref{lipf}) seems to be technical, but such a kind of uniform
coercivity appears to be crucial in order to get a uniform bound for the stationary solutions.

We make the following assumptions on the exchange functions.
\begin{itemize}
	\item The exchange functions $\nu$ and $\mu$ are smooth, nonnegative and compactly supported. Without loss of generality we will assume
	\begin{equation}\label{supportmunu}
	\textrm{supp}(\mu,\nu)\subset(-1,1).	 
	\end{equation}
	\item For the sake of simplicity, we will take $\nub = 1$, as in ~\cite{BRR1}.
\end{itemize}
The parameters $d,D,\mub,\fp$ and the two functions $\nu$ and $\mu$ are fixed once and for all.

\paragraph{Main results of the paper}
The main result of the paper is that spreading in the $x$-direction is indeed uniform in $\e.$ 
Set $c^*_0$ the asymptotic speed of spreading associated to the
initial model (\ref{BRReq2}).
\begin{theorem}\label{uniformspreading}
For $\e>0,$ let us denote $(u_\e,v_\e)$ the solution of system (\ref{RPeps}). There exists $m>0$ such that if $(u_0,v_0)\leq \lp\frac{m}{\mub},m\rp$
we have:
\begin{itemize}
\item  $\forall c>c^*_0,$ $\forall \eta>0,\ \exists T_0,\e_0$ such that 
$\forall t>T_0,\forall \e<\e_0,$ $\displaystyle \sup_{|x|>ct} |u_\e(t,x)|<\eta.$ 
\item  $\forall c<c^*_0,$ $\forall \eta>0,\ \exists T_0,\e_0$ such that 
$\forall t>T_0,\forall \e<\e_0,$ $\displaystyle \sup_{|x|<ct} \lb u_\e(t,x)-\frac{1}{\mub}\rb<\eta.$ 
\end{itemize}
\end{theorem}
The idea of the proof is to show that every solution $(u_\e,v_\e)$ is above some travelling subsolutions 
in finite time. Then, we use the convergence of the spreading speed $c^*_\e$ to $c^*_0$, which yields travelling subsolutions 
at some speed close to $c^*_0.$ Hence, our main tool relies on the following convergence theorem.

\begin{theorem}\label{convsol}
Let $(u,v)$ be the solution of the limit system (\ref{BRReq2})
and $(u_\e,v_\e)$ be the solution of the $\e$-system (\ref{RPeps}) for $\e>0.$ Let $(u_0,v_0)$ be a common initial datum for both systems. Then:
\begin{equation*}
\|(u-u_\e)(t)\|_{L^\infty(\R)}+\|(v-v_\e)(t)\|_{L^\infty(\R^2)}\underset{\e\to0
		}{\longrightarrow} 0. 
\end{equation*}
The above convergence is uniform in every compact set in $t$ included in $(0,+\infty).$
\end{theorem}
Notice that the convergence is global in space, but local in time.

\paragraph{Bibliographical background}
Reaction-diffusion equations of the type 
$$
\partial_t u-d\Delta u=f(u)
$$
have been introduced in the celebrated articles of Fisher ~\cite{fisher} and Kolmogorov, Petrovsky and Piskounov ~\cite{KPP} in 1937.
The initial motivation came from population genetics. The reaction term are that of a logistic law, whose
archetype is $f(u)=u(1-u)$ for the simplest example. In their works in one dimension, 
Kolmogorov, Petrovsky and Piskounov revealed the existence of propagation waves, together with an asymptotic speed of spreading
of the dominating gene, given by $2\sqrt{df'(0)}$. The existence of an asymptotic speed of spreading was generalised in $\R^n$ 
by D. G. Aronson and H. F. Weinberger in ~\cite{AW} (1978). Since these pioneering works, front propagation in 
reaction-diffusion equations have been widely studied. Let us cite, for instance, the works of Freidlin and G\"artner \cite{FG}
for an extension to periodic media, or \cite{W2002}, \cite{BHN1} and \cite{BHN2} for more general domains.
An overview of the subject can be found in ~\cite{BHbook}.

New results on the model under study (\ref{BRReq2}) have been recently proved. Further effects like a drift or 
a killing term on the road have been investigated in \cite{BRR2}. The case of a fractional diffusion on the road was studied 
and explained by the three authors and A.-C. Coulon in \cite{BRRC} and \cite{these_AC}. See \cite{BRRC2} for a summary of the results on this model.
Models with an ignition-type nonlinearity
are also studied by L. Dietrich in \cite{Dietrich1} and \cite{Dietrich2}.

\paragraph{Organisation of the paper}

The second section of the paper is devoted to the convergence of the asymptotic
speed of propagation $c^*_\e$ with $\e$ goes to 0. For this, we will also
investigate some useful convergence result concerning travelling 
supersolutions. The third 
section deals with the convergence of the stationary solutions,
which will be helpful for the control of the long time behaviour 
for Theorem \ref{uniformspreading}. 
Theorem \ref{convsol} is proved in sections \ref{sectionresolvent} and \ref{sectionconvsol}, with an argument from geometric
 theory of parabolic equations. Section \ref{sectionresolvent}, which is the most technical, is devoted to resolvent bounds. 
 They are used in section \ref{sectionconvsol} to prove some convergence properties for the linear systems.
 Then, we use them in a Gronwall argument to deal with the nonlinearity.
 We prove 
Theorem \ref{uniformspreading} in the last section.


\paragraph{Acknowledgements} The research leading to these results has received 
funding from the European Research Council under the European Union’s 
Seventh Framework Programme (FP/2007-2013) / ERC Grant Agreement n.321186 - ReaDi -Reaction-Diffusion Equations, Propagation and Modelling.
I am grateful to Henri Berestycki and Jean-Michel Roquejoffre for suggesting me the
model and many fruitful conversations. I also would like to thank the anonymous referees for their helpful comments.

\section{Asymptotic spreading speed}\label{sectionvitesse}
Let $c^*_0$ be the asymptotic spreading speed associated with the above system (\ref{BRReq2}), and $c^*_\e$ the spreading speed given
by Theorem \ref{spreadingthm} associated with the system (\ref{RPeps}). Under our assumptions,
the main result of this section is
\begin{prop}
	\label{convergencec}
	$c^*_\e$ converges to $c^*_0$ as $\e$ goes to 0, locally uniformly in $d,D,\mub$.
\end{prop}

For the sake of simplicity,
we will consider that $\nu$ is an even function. The general case is similar but heavier.

\subsection{Prerequisite - linear travelling waves and speed of propagation}
All the results given in this subsection have been investigated in \cite{BRR1} by H. Berestycki, J.-M. Roquejoffre, and L. Rossi 
for the local system or in \cite{Pauthier} 
by the author for the nonlocal.
Our purpose is to recall the derivation of the asymptotic speed of spreading for both systems. 
Because $f$ is of KPP-type (that is, satisfies $f(v)\leq\fp v$ for nonnegative $v$),
we are interested in the linearised systems:
\begin{equation}
\label{RPlieps}
\begin{cases}
\partial_t u-D \partial_{xx} u = -\mub u+\int \nu_\e(y)v(t,x,y)dy & x \in \R,\ t>0 \\
  \partial_t v-d\Delta v = f'(0)v +\mu_\e(y)u(t,x)-\nu_\e(y)v(t,x,y) & (x,y)\in \R^2,\ t>0
\end{cases}
\end{equation}
for the nonlocal case, and
\begin{equation}
\label{BRRli}
\begin{cases}
\partial_t u-D \partial_{xx} u= v(t,x,0)-\mub u & x\in\R,\ t>0\\
\partial_t v-d\Delta v=f'(0)v & (x,y)\in\R\times\R^*,\ t>0\\
v(t,x,0^+)=v(t,x,0^-), & x\in\R,\ t>0 \\
-d\left\{ \partial_y v(t,x,0^+)-\partial_y v(t,x,0^-) \right\}=\mub u(t,x)-\nub v(t,x,0) & x\in\R,\ t>0
\end{cases}
\end{equation}
for the local one. This motivates the following definition.
\begin{definition}\label{deftravelling}
 For any of the two systems (\ref{RPlieps})-(\ref{BRRli}), we call a \textbf{linear travelling wave} a 3-tuple
 $(c,\la,\phi)$ with $c>0,$ $\la>0,$ and $\phi\in H^1(\R)$ a positive function such that
 $$
 \begin{pmatrix}
  u(t,x) \\ v(t,x,y)
 \end{pmatrix}
= e^{-\la(x-ct)} \begin{pmatrix}
                        1 \\ \phi(y)
                       \end{pmatrix}
$$
 be a solution of the corresponding linearised system (\ref{RPlieps}) or (\ref{BRRli}). The quantity $c$ is the speed of the exponential travelling wave.
\end{definition}

\begin{remark}\label{remark1}
 From the KPP assumption (\ref{kppassumtions}) on $f$, a linear travelling wave for the linearised system (\ref{RPlieps}) (resp. (\ref{BRRli}))
 provides a supersolution for the nonlinear system (\ref{RPeps}) (resp. (\ref{BRReq2})). This will be a powerful tool 
 to get $L^\infty$ and decay estimates in the sequel.
\end{remark}

The previous definition for travelling waves provides us a helpful characterisation for spreading speed.

\begin{prop}\label{defspreadingspeed}\cite{BRR1,Pauthier}
For any of the systems (\ref{RPeps})-(\ref{BRReq2}), for all $\e>0,$ 
the \textbf{spreading speed $c^*=c^*_0$} or $c^*_\e$ given by Theorem \ref{spreadingthm}
can be defined as follows:
$$
c^*=\inf\{c>0 | \text{ a linear travelling wave with speed $c$ exists}\}.
$$
\begin{enumerate}
 \item If $D\leq 2d,$ then $c^*_0=c^*_\e=c_{KPP}=2\sqrt{d\fp}.$
 \item If $D>2d,$ then $c^*_0,c^*_\e>c_{KPP}$ and the infimum is reached by a linear travelling wave, denoted $(c_0^*,\la_0^*,\phi^*)$
 or $(c_\e^*,\la_\e^*,\phi^*_\e).$
\end{enumerate}
\end{prop}

Proposition \ref{defspreadingspeed} provides the construction of $c^*$ thanks to a nonlinear eigenvalue
problem. We give an outline of a proof in both cases to make the sequel easier to read. We focus only on the case 
$D>2d,$ the other being trivial.

\paragraph{Resolution for the local case}
Inserting the definition supplied by Proposition \ref{defspreadingspeed} into (\ref{BRRli}), we obtain the following system
in $(c,\la,\phi):$
\begin{equation}\label{clambdaphi0}
 \begin{cases}
-D\la^2+\la c+\mub=\phi(0) \\
-d\phi''(y)+\lp\la c-d\la^2-\fp\rp\phi(y)=0 \qquad y\in\R^*\\
\phi(0^+)=\phi(0^-),\ \phi\geq0,\ \phi\in H^1(\R), \\
 -d\lp\phi'(0^+)-\phi'(0^-)\rp=\mub-\phi(0).
 \end{cases}
\end{equation}
From now, we set 
$$
P(\la)=\la c-d\la^2-\fp,\ \la_2^\pm(c)=\frac{c\pm\sqrt{c^2-c_{KPP}^2}}{2d}
$$
and
$$
\MD=\left\{ (c,\la),\ c>c_{KPP} \textrm{ and } \la\in(\la_2^-(c),\la_2^+(c))\right\}.
$$
Hence, the existence of a linear travelling wave is equivalent to the following system in $(c,\la,\phi(0))$
provided that $c>c_{KPP}$ and $\la\in[\la_2^-(c),\la_2^+(c)]$:
\begin{equation}\label{systemelaphilocal}
 \begin{cases}
 -D\la^2+\la c+\mub=\phi(0) \\
 \phi(0)=\frac{\mub}{1+2\sqrt{dP(\la)}}.
 \end{cases}
\end{equation}
The first equation of (\ref{systemelaphilocal}) gives the graph of a function
\begin{equation}\label{gamma1}
 \la\mapsto\Psi_1^0(c,\la):=-D\la^2+c\la+\mub
\end{equation}
which is intended to be equal to $\phi(0),$ 
provided $(c,\la,\phi)$ defines a linear travelling wave. Let us denote $\Gamma_1$ its graph, depending on $c,$ 
in the $(\la,\Psi_1^0(\la))-$plane.

The second equation of (\ref{systemelaphilocal}) gives the graph of a function
\begin{equation*}
\Psi_2^0 : \left\{ \begin{array}{ccl}
\overline{\MD} & \longrightarrow & \R \\
(c,\la) & \longmapsto &\frac{\mub}{1+2\sqrt{dP(\la)}}.
\end{array} \right.
\end{equation*}
Let us denote $\Gamma_2^0$ its graph in the same plane, still depending on $c.$
Hence, (\ref{systemelaphilocal}) amounts to looking for the first $c\geq c_{KPP}$ such that the two graphs
$\Gamma_1$ and $\Gamma_2^0$ intersect. 

It was shown that there exists $(c^*_0,\la^*_0)\in\MD$ and an exponential function 
$$\phi^* : y \mapsto \phi^*(0)e^{-\sqrt{P(\la_0^*)}|y|}$$
such that $(c^*_0,\la^*_0,\phi^*)$ defines a linear travelling wave for (\ref{BRRli}), and $c^*_0$ is
the first $c$ such that (\ref{clambdaphi0}) admits a solution. Hence, $c^*_0$ is the speed defined by Proposition \ref{defspreadingspeed}.

\paragraph{Resolution for the nonlocal case}
In this case, Proposition \ref{defspreadingspeed} yields the following system in $(c,\la,\phi).$
\begin{equation}\label{clambdaphieps}
 \begin{cases}
    -D\la^2+c\la+\mub = \int_\R \nu_\e(y)\phi(y)dy \\
    -d\phi''(y)+\lp c\la-d\la^2-\fp+\nu_\e(y)\rp\phi(y) = \mu_\e(y),\ \phi\in H^1(\R).
 \end{cases}
\end{equation}
Once again, the first equation of (\ref{clambdaphieps}) gives the function $\Psi_1^0$ defined by (\ref{gamma1}), which is intended to be equal to $\int \nu_\e\phi$
provided $(c,\la,\phi)$ defines a linear travelling wave for (\ref{RPlieps}). The second equation of (\ref{clambdaphieps}) defines implicitly the
function
\begin{equation*}
\Psi_2^\e : \left\{ \begin{array}{ccl}
\MD & \longrightarrow & \R \\
(c,\la) & \longmapsto & \int \nu_\e(y)\phi(y;\e,c,\la)dy
\end{array} \right.
\end{equation*}
where $\phi(.;\e,c,\la)$ is the unique solution in $H^1(R)$ of
\begin{equation}
\label{gamma2e}
-d\phi''(y)+(\la c-d\la^2-f'(0)+\nu_\e(y))\phi(y)  =  \mu_\e(y),\qquad y\in\R.
\end{equation}
For fixed $c,$ we denote $\Gamma_2^\e$ the graph of $\Psi_2^\e$ in the $(\la,\Psi_2^\e(\la))-$plane.
$\Psi_2^\e$ is smooth in $\MD$ and can be continuously extended in $\overline{\MD}.$ It has also been proved that
$\phi,$ hence $\Psi_2^\e$ is decreasing in $c.$
It was shown that for all $\e>0,$ there exists $(c^*_\e,\la^*_\e)\in\MD$ and $\phi^*_\e=\phi(.;\e,c^*_\e,\la^*_\e)$ solution
of (\ref{gamma2e}) such that $(c^*_\e,\la_\e^*)$ $\Gamma_1^0$ and $\Gamma_2^\e$ intersect, and $c_\e^*$ is the first $c$ such that it occurs.
The main ingredients of this proof are:
\begin{itemize}
 \item the functions $c\mapsto\Psi_1^0(c,\la)$ and 
$c\mapsto\la_2^-(c)$ are respectively increasing and decreasing;
\item the function $\la\mapsto\Psi_2^\e(c,\la)$ is strictly concave.
\end{itemize}
These behaviours summed up in Figure \ref{2graphs}.
\begin{figure}[!ht]
   \centering
   \psset{unit=0.6}
\def\mub{\overline{\mu}}

\begin{pspicture}(-3,-1)(13,9)
\psline{->}(-2,0)(12.7,0)
\psline{->}(0,-1)(0,8.5)
\psparabola[linewidth=0.1,linecolor=red](-2,0)(2,8)
\psellipticarc[linewidth=0.1,linecolor=red](10,6)(2.5,3.5){180}{360}

\psline[linestyle=dashed](0,6)(12.5,6)
\uput{0.2}[180](0,6){$\mub$}
\uput{0.2}[0](12.7,0){$\lambda$}
\uput{0.2}[-90](7.5,0){$\lambda_2^-(c)$}
\uput{0.2}[-90](12.5,0){$\lambda_2^+(c)$}
\psline[linestyle=dashed](2,8)(2,0)
\psline[linestyle=dashed](4,6)(4,0)
\psline[linestyle=dashed](7.5,6)(7.5,0)
\psline[linestyle=dashed](10,2.5)(10,0)
\psline[linestyle=dashed](12.5,6)(12.5,0)
\uput{0.2}[-90](10,0){$\frac{c}{2d}$}
\uput{0.2}[-90](2,0){$\frac{c}{2D}$}
\uput{0.2}[-90](4,0){$\frac{c}{D}$}
\uput{0.2}[90](10,2.5){$\Gamma_2^\e$}
\uput{0.2}[90](2,8){$\Gamma_1$}
\uput{0.2}[180](0,8.5){$\Psi_{1,2}$}
\psline[linewidth=0.1]{->}(4.2,5)(5.2,5)
\psline[linewidth=0.1]{->}(8.6,3.5)(7.6,3.5)
\psline[linewidth=0.1]{->}(11.4,3.5)(12.4,3.5)
\end{pspicture}
   \caption{\label{2graphs}representation of $\Gamma_1$ and $\Gamma_2^\e,$ behaviours as $c$ increases}
\end{figure}
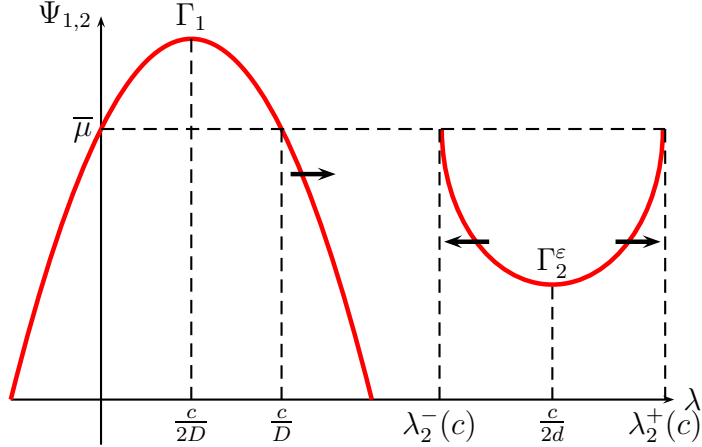
The main purpose of the rest of this section will be to show that the curve $\Gamma_2^\e$ converges locally uniformly in $(c,\la)$
to $\Gamma_2^0$ as $\e$ goes to 0.

\subsection{Convergence of the spreading speed}

\paragraph{Uniform boundedness of $\phi$ in $C^1(\R)$}
\begin{lem}\label{borneuniformephi}
Let us consider $\phi:=\phi(y;\e,c,\la)$ the unique solution of (\ref{gamma2e}) for $\e>0$ 
and $(c,\la)$ in $\MD,$ that is to say $c>c_{KPP}$ and $\la$ in $]\la_2^-(c),\la_2^+(c)[.$
Then the family $\lp\|\phi\|_{L^\infty(y)}+\|\phi'\|_{L^\infty(y)}\rp_\e$ is uniformly bounded in $\e>0$ and every compact set on 
$\overline{\MD}.$
\end{lem}

\begin{proof}
We consider $\phi=\phi(y;\e,\la,c)$ defined for $\e>0$, $(c,\la)$ in $\MD,$ and $y\in\R$. We know
that $\phi$ can be continuously extended to $\overline{\MD}.$
Moreover, from the elliptic maximum principle, it is easy to see that $\phi>0$ on $\R,$ for all admissible parameters.
Considering the hypotheses (\ref{supportmunu}) on $\mu, \nu$, we get that: 
\begin{itemize}
\item $\forall \e>0,\ supp(\nu_\e,\mu_\e)\subset(-\e,\e)$ ;
\item $y\mapsto \phi(y)$ is even.
\end{itemize}
Hence there exists $K=K(\e,c,\la)$ such that 
\begin{equation}\label{referee2}
 \forall |y|>\e,\ \phi(y;\e,c,\la)=K(\e,c,\la)e^{-\sqrt{\frac{P(\la)}{d}}|y|}.
\end{equation}

\textit{Step 1}. It is enough to bound $K$ to get the uniform boundedness of $\|\phi(\e,c,\la)\|_{L^{\infty}(\R)}$.
Indeed, using the scaling $\xi=\frac{y}{\e}$, with $\psi(\xi)=\phi(y)$, the function $\psi$ satisfies
\begin{equation}
\label{eqrescal1}
\begin{cases}
-d\psi''(\xi)+(\e^2 P(\la)+\e\nu(\xi))\psi(\xi) =  \e\mu(\xi) \\
\psi(\pm 1)=Ke^{-\sqrt{\frac{P(\la)}{d}}}.
\end{cases}
\end{equation}
Now, let us recall that $\mu,\nu$ are continuous and compactly supported. Then, 
by the Harnack inequality (theorem 8.17 and 8.18 in ~\cite{GT} for instance), there
exist $C_1,C_2\geq 0$, independent of $\e,c,\la$, such that
$$
\underset{[-1,1]}{\sup} \psi \leq C_1(\underset{[-1,1]}{\inf}\psi + C_2),
$$
which gives immediately
$$
\underset{\R}{\sup}\ \phi \leq C_1(K(\e,c,\la) + C_2).
$$

\textit{Step 2}. Let us prove a uniform bound for $K$. Set $c_1\in]c_{KPP},+\infty[$, and assume by contradiction that
\begin{equation}\label{absurde}
 \lim_{\e\to0}\sup \left\{ K(\e,c,\la),\ (c,\la)\in\overline{\MD},c\leq c_1\right\}=+\infty.
\end{equation}
That is, there exist $(\e_n)_n,(\la_n)_n,(c_n)_n$ with $\e_n\to 0$ and $c\leq c_1,$ $\la_2^-(c_1)\leq\la\leq\la_2^+(c_1)$
such that 
$$
K(\e_n,\la_n,c_n):=K_n\underset{n\to \infty}{\longrightarrow} +\infty.
$$
Set $$\phit_n = \frac{\phi(.;\e_n,\la_n,c_n)}{K_n}.$$ 
The function $\phit_n$ satisfies
\begin{equation}
\label{phitilde}
\begin{cases}
-d\phit_n''+(P(\la_n)+\nu_{\e_n})\phit_n = \frac{\mu_{\e_n}}{K_n} \ y\in\R \\
\phit_n(y) = e^{-\sqrt{\frac{P(\la_n)}{d}}|y|},\ \forall |y|>\e_n.
\end{cases}
\end{equation}
Again by the Harnack inequality, $(\|\phit_n\|_\infty)_n$ is bounded, and $\phit_n$ is positive by the elliptic maximum principle. 
Integrating (\ref{phitilde}) between $-\infty$ and $y$ gives 
$$
d\phit_n'(y) = P(\la_n)\int_{-\infty}^y\phit_n+\int_{-\infty}^y \nu_{\e_n} \phit_n-\frac{1}{K_n}\int_{-\infty}^y \mu_{\e_n},
$$
so
\begin{align*}
 d\|\phit_n'\|_\infty  & \leq P(\la_n)\int_\R e^{-\sqrt{\frac{P(\la_n)}{d}}|y|}dy+P(\la_n)\int_{-\e}^{+\e}\phit_n(y)dy+
 \|\phit_n\|_\infty+\frac{\mub}{K_n} \\
  & \leq 2\sqrt{dP(\la_n)}+\|\phit_n\|_\infty\lp1+2\e P(\la_n)\rp+\frac{\mub}{K_n}.
\end{align*}
Hence $(\phit_n)_n$ is uniformly Lipschitz.
Specializing to $y=1$, we get: 
\begin{equation}\label{inegalite}
-d\phit_n'(1) = \frac{\mub}{K_n}-P(\la_n)\int_{-\infty}^1\phit_n-\int_{-\infty}^1 \nu_{\e_n} \phit_n.
\end{equation}

But on the other hand, we have:
\begin{itemize}
\item $-d\phit_n'(1) = \sqrt{dP(\la_n)} e^{-\sqrt{\frac{P(\la_n)}{d}}}\geq 0$ from (\ref{referee2}) ;
\item $\frac{\mub}{K_n}\to 0$ as $n\to\infty$ by assumption (\ref{absurde}) ; 
\item $-P(\la_n)\int_{-\infty}^1\phit_n\leq 0$ ; 
\item $-\int_{-\infty}^1 \nu_{\e_n} \phit_n\to -1$ as $n\to \infty$. Indeed, the sequence $\nu_{\e_n}$ tends to the Dirac measure, we have
$\phit_n(\pm\e_n)=1+O(\e_n)$ from (\ref{phitilde}), and $(\phit_n)_n$
is uniformly Lipschitz.
\end{itemize}
For $n\to\infty$, this contradicts (\ref{inegalite}) since the left term is nonnegative and the right term tends  to a negative limit.
Hence, there is a contradiction. That is, $K(\e,c,\la)$ is bounded for $(c,\la)\in\overline{\MD}$ satisfying $c\leq c_1.$
Recall that $c\mapsto \phi$ is nonincreasing,
and $\phi$ is uniformly bounded as $\e\to 0$ in $(c,\la)\in\overline{\MD}.$

\textit{Step 3}. Boundedness of $\|\phi'\|_\infty$ with $\e\to 0$. We integrate (\ref{gamma2e}) from $-\infty$ to $y$ which 
gives: 
$$
d\phi'(y) = P(\la)\int_{_\infty}^y \phi + \int_{_\infty}^y \nu_\e \phi - \int_{_\infty}^y \mu_\e.
$$
Now, the explicit formula for $\phi$ and its uniform boundedness obtained in step 2 yields: 
$$
d\|\phi'\|_{\infty} \leq \|\phi\|_{\infty}\lp 1+2\sqrt{dP(\la)}+2\e\rp+\mub
$$
and the family $(\|\phi(\e,c,\la)\|_{\infty})$ is equicontinuous for all $\e$ close enough to $0$
and for all compact set in $\overline{\MD}.$ This implies the uniform boundedness of $\phi$ in $C^1(\R)$.
\end{proof}

\paragraph{Convergence of $\phi(.;\e)$ with $\e\to 0$, continuity of $\Psi_2^{\e}$.}
From Lemma \ref{borneuniformephi} the set $\left(\phi(.;\e) \right)_\e$ is included in $C_b(\R)$ and equicontinuous, locally uniformly in $c,\la$.
The Arzel\`a-Ascoli theorem (combined with Cantor's diagonal argument) yields the existence of a sequence $(\e_n)_n \subset \R,\ \e_n\to 0$ and
a function $\phi_0\in C_b(\R)$ such that $\phi(\e_n)\underset{n\to \infty}{\longrightarrow}\phi_0$ uniformly on compact sets.
Passing to the limit in (\ref{gamma2e}), we obtain that $\phi_0$ satisfies 
\begin{equation}\label{eqlimite}
-d\phi_0''+\phi_0(P(\la)+\d_0) = \mub\d_0
\end{equation}
in the distribution sense.
Moreover, $K(\e_n,c,\la)\to K_0(c,\la)=\phi_0(0;c,\la)$ with $n\to\infty$,
and $$\phi_0(y)=K_0\exp(-\sqrt{\frac{P(\la)}{d}}|y|).$$

It remains to show the uniqueness of the limit function. Let $\phi_1$ be another accumulation point for $(\phi(.;\e))_\e$. Then $\phi_1$ also satisfies
(\ref{eqlimite}) in the distribution sense, and $\phi_1(y)=\phi_1(0)\exp(-\sqrt{\frac{P(\la)}{d}}|y|)$. Then $\psi=\phi_0-\phi_1$ is a solution of
\begin{equation*}
-d\psi''+\psi(P(\la)+\d_0)=0.
\end{equation*}
Now, let us consider $(\psi_n)_n\subset\mathcal{D}(\R)$ such that $\psi_n\to\psi$ uniformly on every compact and $\psi_n'\to\psi'$ on every compact
of $\R\backslash\{0\}.$ We get 
$$
d\|\psi'\|^2_{L^2}+P(\la)\|\psi\|^2_{L^2}+\psi(0)^2=0,
$$
and $\psi\equiv0$. So $\phi_0$ is the unique accumulation point of $(\phi(.;\e))_\e$, hence $\phi(.;\e)\underset{\e\to0}{\longrightarrow}\phi_0$ 
uniformly on compact sets. Hence we have proved the following lemma.

\begin{lem}\label{convergencephi}
Let us consider $\phi$ as a function of $\e,c,\la,$ extended to $\phi_0$ 
for $\e=0.$ Then $\phi(.;\e,c,\la)$ is continuous from $[0,1]\times \overline{\MD}$
to $C^0(\R)$ and bounded in $C^1$ as a $y$-function on every compact set on it.
\end{lem}

It is now easy to see that the function $\Psi_2^{\e}$ converges continuously in $\la,c$, to a
limit function 
\begin{equation*}
\Psi_2^l: \left\{ \begin{array}{ccl}
]\la_2^- , \la_2^+[ & \longrightarrow & \R \\
\la & \longmapsto & \phi_0(c,\la)
\end{array} \right.
\end{equation*}
and the curve $\Gamma_2^\e$ converges to the graph of $\Psi_2^l$. Let us denote it $\Gamma_2^l$.

\paragraph{The limit curve $\Gamma_2^l$.} Now let us show that $\Gamma_2^l=\Gamma_2^0.$ The limit function
$\phi_0$ is of the form $\phi_0(y)=K_0\exp(-\sqrt{\frac{P(\la)}{d}}|y|)$ and satisfies (\ref{eqlimite}) in the distribution sense.
Applying the two sides of (\ref{eqlimite}) to $\phi_0$ (or, to be strictly rigorous, to a sequence $\lp\phi_n\rp$ that tends to $\phi_0$), 
we get
$$
d\|\phi_0'\|^2_{L^2}+P(\la)\|\phi_0\|^2_{L^2}+\phi_0(0)^2=\mub \phi_0(0).
$$
Using the explicit formula for $\phi_0$, we obtain: 
$$
2\sqrt{dP(\la)}\phi_0(0)^2+\phi_0(0)^2 = \mub\phi_0(0).
$$
Hence, because 0 is not a solution,
$$
\phi_0(0)=\frac{\mub}{1+2\sqrt{dP(\la)}}.
$$
Hence, $\Psi_2^l=\Psi_2^0$ and $\Gamma_2^l=\Gamma_2^0$ which concludes the proof of Proposition \ref{convergencec}.
\qed

\begin{remark}
 Actually, the spreading speed in the local case was not devised exactly with the system (\ref{systemelaphilocal}), but with the following system.
\begin{equation*}
\begin{cases}
-D\la^2+c\la = & \frac{\mub}{1+2d\b}-\mub \\
-d\la^2+c\la = & f'(0)+d\b^2.
\end{cases}
\end{equation*}
But, setting $\Phi(\b) = \frac{\mub}{1+2d\b},$ they are equivalent.
\end{remark}

%

\section{Convergence of the stationary solutions}

We have already showed in \cite{Pauthier} that, 
for all $\e>0,$ there exists a unique positive and bounded stationary solution of 
(\ref{RPeps}), which will be  denoted $(U_\e,V_\e).$ Moreover, this stationary solution is $x-$independent and
satisfies $\displaystyle \lim_{y\to\pm\infty}V_\e(y)=1.$ The corresponding equation is
\begin{equation}
\label{stationaryeps}
\begin{cases}
\mub U_\e = \int \nu_\e(y)V_\e(y)dy \\
-dV_\e''(y) = f(V_\e)+\mu_\e(y)U_\e-\nu_\e(y)V_\e(y).
\end{cases}
\end{equation}
Solutions of (\ref{stationaryeps}) depend only of the $y$-variable, hence $U$ is 
constant and $V$ is entirely determined by the following integro-differential equation
\begin{equation}
\label{Veqeps}
-dV''_\e(y)=f(V_\e)+\frac{\mu_\e(y)}{\mub}\int \nu_\e(z)V_\e(z)dz-\nu_\e(y)V_\e(y).
\end{equation}
The main result of this section lies in the next proposition.

\begin{prop}
\label{convVeps}
$(V_\e)_\e$ converges uniformly to the constant function 1 when $\e$ goes to 0.	
\end{prop}

Of course, this implies that $U_\e\to\frac{\nub}{\mub}$ as $\e\to0.$ As a 
result, the stationary solutions of (\ref{RPeps}) converge to the stationary solutions
of (\ref{BRReq2}) which were identified in \cite{BRR1}.

The difficulties lie in the singularity in $(-\e,\e).$ Outside this interval, 
$V_\e$ satisfies
\begin{equation}
	\label{portraitphase}
	\begin{cases}
	-dV_\e''=f(V_\e) \\
	0<V_\e<+\infty, \ V_\e(\pm\infty)=1.
	\end{cases}
\end{equation}
As $(V_\e,V_\e')=(1,0)$ is a saddle point for the system (\ref{portraitphase}), it is easy to see
that solutions of (\ref{portraitphase}) belong to one of the two integral curves that 
tend to $(V_\e,V_\e')=(1,0).$ We can also notice that $V_\e(\e)>1$ (resp. $<1$) implies that 
$V_\e$ is decreasing (resp. increasing) on $(\e,+\infty).$ Thus important estimates have to be found inside $(-\e,\e).$ 
From now and without lack of generality, we will assume $d=1.$

\paragraph{A uniform $L^\infty$ boundary for $(V_\e).$}
Set $z=\frac{y}{\e}$ and
$
W_\e(z):=V_\e(y).
$
Then, we have $\|V_\e\|_\infty=\|W_\e\|_\infty,$ $\|V_\e'\|_\infty=\frac{1}{\e}\|W_\e'\|_\infty,$ and $W_\e$ satisfies in $(-1,1)$
\begin{equation}
 \label{Weps}
 -W_\e''=\e^2f(W_\e)+\e\frac{\mu(z)}{\mub}\int \nu(s)W_\e(s)ds-\e\nu(z)W_\e(z).
\end{equation}

\textit{Step 1}. Lower bound for $W_\e(1).$ From (\ref{portraitphase}), there exists at least a point $z_\e\in(-1,1)$ such that
$W_\e'(z_\e)=0.$ Hence, 
\begin{equation}\label{Wprime}
\forall z\in (-1,1),\ W_\e'(z)=\int_{z_\e}^z W_\e''(s)ds.
\end{equation}
Integrating (\ref{Weps}) yields the following rough bound, for all $\e>0:$
\begin{equation}\label{Wprimeeps}
 \forall z\in(-1,1),\ |W_\e'(z)|\leq 2\e \lp\|\nu\|_\infty+1+\e\|f\|_{Lip}\rp\|W_\e\|_{L^\infty(-1,1)}.
\end{equation}
Let us set $K=4\lp\|\nu\|_\infty+\nub+\|f\|_{Lip}\rp$ and we get the following estimate for $W_\e(1)=V_\e(\e):$
\begin{equation}
 \label{Weps1}
 \left| \frac{W_\e(1)-\|W_\e\|_{L^\infty(-1,1)}}{\|W_\e\|_{L^\infty(-1,1)}}\right| \leq K\e.
\end{equation}

\textit{Step 2.} Lower bound for $W_\e'(1).$ Using once again (\ref{Wprime}), we have
\begin{equation}
 \label{Wprime1}
 W_\e'(1)=\e\int_{z_\e}^1 \nu(s)W_\e(s)ds-\e^2\int_{z_\e}^1 f(W_\e(s))ds-\lp\frac{\e}{\mub}\int_{-1}^1 \nu(s)W_\e(s)ds\rp\int_{z_\e}^1 \mu(s)ds.
\end{equation}
The first term in the right handside in (\ref{Wprime1}) is nonnegative, hence 
$W_\e'(1)\geq -\e(\nub+\e\|f\|_{Lip})\|W_\e\|_\infty$ and, in the $y$-variable : 
\begin{equation}
 \label{Vprime1}
 V_\e'(\e)\geq -\|V_\e\|_\infty(\nub+\e\|f\|_{Lip}).
\end{equation}

\textit{Step 3}. Proof of the boundedness by contradiction. Let us suppose that
$$\lim_{\e\to0}\sup \|V_\e\|_{L^\infty} = +\infty.$$
That is, there exists a sequence $\e_n\to0$ such that
$\displaystyle
\underset{\R}{\sup }V_{\e_n}=\underset{(-\e_n,\e_n)}{\sup }V_{\e_n}\underset{n\to\infty}{\longrightarrow}+\infty.
$
From (\ref{Weps1}), we have $V_{\e_n}(\e_n)\geq \|V_{\e_n}\|_\infty(1-K\e_n).$ Now, recall that on $(\e_n,+\infty),$
the function $V_{\e_n}$ satisfies (\ref{portraitphase}). Multiply it by $V_{\e_n}'$ and integrate, we get
$V_{\e_n}'^2(\e_n)=-2F(V_{\e_n}(\e_n))$ where $F(t):=\int_1^t f(s)ds$ is an antederivative of $f.$
Considering the hypotheses on $f$ this gives
$$
V_{\e_n}'(\e_n)\underset{n\to\infty}{\sim}-\sqrt{\|f\|_{Lip}}V_{\e_n}(\e_n).
$$
From hypothesis (\ref{lipf}) on $f,$ $\|f\|_{Lip}>2$ and we get a contradiction with (\ref{Vprime1}) and (\ref{Weps1}). As a result,
$\lp\|V_\e\|_{L^\infty(\R)}\rp_\e$ and $\lp\|V'_\e\|_{L^\infty(\R)}\rp_\e$ are uniformly bounded as $\e$ goes to 0.

\paragraph{Convergence of the stationary solutions} From the previous paragraph and Ascoli's Theorem, $(V_\e)_\e$ 
admits at least one accumulation point, let say $V_0$, and the convergence is uniform on every compact set, thus uniform
on $\R$ (from the monotonicity of $V_\e$ outside $(-\e,\e),$ or even a diagonal argument). So $V_0$ is continuous, bounded, 
tends to 1 at infinity.
Passing to the limit in (\ref{Veqeps}), it satisfies in the distribution sense
$$
-V_0''=f(V_0)+\nub\d_0(V_0(0)-V_0).
$$
As the support of the Dirac distribution is reduced to $\{0\}$ , and because of the continuity of $V_0,$ it satisfies
in the classical sense
\begin{equation}\label{groumf}
 \begin{cases}
  -V_0''(y)=f(V_0(y)), \ y\in \R \\
  V_0(\pm\infty)=1.
 \end{cases}
\end{equation}
The only solution of (\ref{groumf}) is $V_0\equiv1.$ Hence,
the set $(V_\e)_\e$ admits only one accumulation point, so $V_\e\to 1$ as $\e$ goes to 0 uniformly on $\R,$ and the proof
of Proposition \ref{convVeps} is complete.
\qed

This convergence allows us to assert that there exist $0<m<M<+\infty$ such that 
\begin{equation}\label{uniformsupersol}
 m < V_\e(y)< M,\ \forall \e>0,\ \forall y\in\R.
\end{equation}
Thus, any solution of (\ref{RPeps}) starting from an initial datum $(u_0,v_0)\leq(\frac{\nub}{\mub}m,m)$
will remain below $M,$ which gives a uniform supersolution in $\e.$


\section{Uniform bounds on the resolvents}\label{sectionresolvent}
Consider the two linear models
\begin{equation}
\label{RPepsli}
\begin{cases}
\partial_t u-D \partial_{xx} u = -\mub u+\int \nu_\e(y)v(t,x,y)dy & x \in \R,\ t>0 \\
\partial_t v-d\Delta v = \mu_\e(y)u(t,x)-\nu_\e(y)v(t,x,y) & (x,y)\in \R^2,\ t>0
\end{cases}
\end{equation}
and
\begin{equation}
\label{BRReq2li}
\begin{cases}
\partial_t u-D \partial_{xx} u= v(x,0,t)-\mub u & x\in\R,\ t>0\\
\partial_t v-d\Delta v=0 & (x,y)\in\R\times\R^*,\ t>0\\
v(t,x,0^+)=v(t,x,0^-), & x\in\R,\ t>0 \\
-d\left\{ \partial_y v(t,x,0^+)-\partial_y v(t,x,0^-) \right\}=\mub u(t,x)-\nub v(t,x,0) & x\in\R,\ t>0.
\end{cases}
\end{equation}
The general goal is to give a uniform bound on the 
difference between solutions of the two above linear systems. We choose a sectorial 
operators approach to get an integral representation of the semigroups generated by (\ref{RPepsli}) and (\ref{BRReq2li}).
They both can be written in the form
\begin{equation}\label{edpedo}
\partial_t \begin{pmatrix}
u \\
v
\end{pmatrix}
=L\begin{pmatrix}
u \\
v
\end{pmatrix}
\end{equation}
where $L=L_\e$ in (\ref{RPepsli}) and $L=L_0$ in (\ref{BRReq2li}) are linear unbounded operators defined by
$$
L_\e \ : \ \left\{ \begin{array}{rcl}
\mathcal{D}(L_\e)\subset X & \longrightarrow & X \\
\begin{pmatrix}u\\v\end{pmatrix} & \longmapsto & \begin{pmatrix}D\partial_{xx}u-\mub u+\int \nu_\e v\\
d\Delta v+\mu_\e u-\nu_\e v\end{pmatrix} 
\end{array}\right.
$$
$$
L_0 \ : \ \left\{ \begin{array}{rcl}
\mathcal{D}(L_0)\subset X & \longrightarrow & X \\
\begin{pmatrix}u\\v\end{pmatrix} & \longmapsto & \begin{pmatrix}D\partial_{xx}u-\mub u+\nub v(.,0)\\
d\Delta v\end{pmatrix} 
\end{array}\right.
$$
with $X$ the space of continuous functions decaying to 0 at infinity $\MC_0(\R)\times\MC_0(\R^2)$. 
The domains are those of the Laplace operator, with 
exchange conditions included in $\mathcal{D}(L_0).$

We recall the definition of a sectorial operator:
\begin{definition}\label{sectorialdef}
	A linear operator $A : \mathcal{D}(A)\subset X \to X$ is {\bf sectorial} 
	if there are constants $\vp\in\R,$ $\theta\in (\frac{\pi}{2},\pi),$ and $M>0$ such that
	\begin{equation}\label{defsectorial}
	\begin{cases}
	(i) \qquad \rho(A)\supset S_{\theta,\vp}:=\{\la\in\C:\la\neq\vp,|\arg(\la-\vp)|<\theta\} \\
	(ii) \qquad \|R(\la,A)\|_{\mathcal{L}(X)}\leq \frac{M}{|\la-\vp|},\ \la\in S_{\theta,\vp} 
	\end{cases}
	\end{equation}
\end{definition}

\begin{prop}
	\label{sectorial}
	Let $\vp>\max(2,3\mub),$ $\frac{\pi}{2}<\theta<\frac{3\pi}{4}.$ Then $(L_\e,\mathcal{D}(L_\e))$ and $(L_0,\mathcal{D}(L_0))$ are sectorial with
	sector $S_{\theta,\vp},$ $\forall \e>0.$
\end{prop}
 Let us denote $M_0,M_\e$ the corresponding constant for the norm of
	the resolvents.
The proof for $L_\e$ is quite standard and omitted here. There is a general approach of the theory in \cite{henry}. A proof for 
$L_0$ can be found in \cite{these_AC}. Assumptions on $\vp$ and $\theta$ are only technical and can be improved.

From Proposition \ref{sectorial}, we know (see \cite{henry} for instance) that, for all $t>0,$ solutions of (\ref{edpedo}) have the form
$$
\begin{pmatrix} u(t) \\ v(t) \end{pmatrix} = 
e^{tL}\begin{pmatrix} u_0 \\ v_0 \end{pmatrix}
$$
where 
\begin{equation}\label{semigroup}
e^{tL}=\frac{1}{2\pi i}\int_{\Gamma_{r,\vartheta}}e^{t\la}R(\la,L)d\la	
\end{equation}
for any $r>0, \vartheta\in(\frac{\pi}{2},\theta),$ where $\Gamma_{r, \vartheta}:=\{\la,|\arg(\la-\vp-r)|=\vartheta\}$ is
a counterclockwise oriented curve which encloses the spectrum of $L,$ and will be denoted $\Gamma$ when there is no possible confusion.
Let us fix from now and for all $r>0$ and the angle $\vartheta$ as above. A parametrisation
of $\Gamma_{r, \vartheta}$ is then given by
$s\in\R\mapsto \vp+r+s e^{i\vartheta. \sgn(s)}.$

For $\lp U,V\rp\in X,$ we will denote in this section:
 \begin{equation*} 
 \begin{cases}
 (u,v)=R(\la,L_0)(U,V) \\
 (u_\e,v_\e)=R(\la,L_\e)(U,V)
 \end{cases}
 \end{equation*}
  that is 
 \begin{equation}\label{avtfourier0}
 \begin{cases}
 (\la+\mub) u- D\partial_{xx}u = v(.,0) + U \qquad  x\in\R \\
 \la v-\Delta v = V \qquad (x,y)\in\R\times\R^* \\
 v(x,0^+)=v(x,0^-)\qquad x\in\R\\
 -(\partial_y v(x,0^+)-\partial_y v(x,0^-)) = \mub u-v(x,0) \qquad x\in\R
 \end{cases}
 \end{equation}
 and
 \begin{equation}\label{avtfouriereps}
\begin{cases}
(\la+\mub) u(x)-D\partial_{xx}u(x) & = \int\nu_\e(y)v(x,y)dy + U(x) \\
(\la+\frac{1}{\e}\nu(\frac{y}{\e})) v(x,y) -\Delta v(x,y) & = \frac{1}{\e}\mu(\frac{y}{\e})u(x)+V(x,y).
\end{cases}
 \end{equation}
The purpose of this section is to give some estimates on the resolvents, that is on the solutions of (\ref{avtfourier0})
and (\ref{avtfouriereps}). They are given in Lemma \ref{grandslambda} and Corollary \ref{corl1}.
 
Lastly, let us recall (see \cite{henry} or \cite{Lunardi}) that the Laplace operator is also sectorial, with a sector 
strictly containing $S_{\theta,\vp}.$ Thus, there exists a constant $M>0$ such that for $d\in\{1,2\},$
\begin{equation}
\label{boitenoire}
\forall w\in \MC_0(\R^d),\ \forall \la\in S_{\theta,\vp},\ \|w\|_\infty \leq \frac{M}{|\la|}\|\Delta w-\la w\|_\infty.
\end{equation}


\subsection{Large values of $\lb\la\rb$}
\begin{lem}
	\label{grandslambda}
	There exist $\e_0>0$ and a constant $C_1$ depending only on $D,$ $\mub$ and $\vartheta$ such that for all positive $\e<\e_0,$ for $\beta>\frac{1}{2},$  
	$$\textrm{if }\la\in \Gamma_{r,\vartheta} \textrm{ and } |\la|>\e^{-\beta} \textrm{, then }\|R(\la,L_\e)\|\leq C_1\max(\e^\beta,\e^{2\beta-1}). $$
\end{lem}

\begin{proof}
	Let $(U,V)\in X,$ and $(u_\e,v_\e)=R(\la,L_\e)(U,V)$ be a solution of (\ref{avtfouriereps}).
	Assumptions on $\la$ imply for $\e$ small enough that 
	$|\Im {\la}|>\frac{1}{2}\sin(\vartheta)\e^{-\beta}.$ Thus, $\nu$ being a real nonnegative function, 
	$\la+\frac{1}{\e}\nu(\frac{y}{\e})\in S_{\theta,\vp}, \forall y\in\R$ and 
	$$
	\left|\la+\frac{1}{\e}\nu(\frac{y}{\e})\right|\geq\Im{\la+\frac{1}{\e}\nu(\frac{y}{\e})}>\frac{1}{2}\sin(\vartheta)\e^{-\beta}.
	$$
	In the same way, we get a similar lower bound for $\left|\la+\mub\right|.$
	Now we use (\ref{boitenoire}) in (\ref{avtfouriereps}) with the above estimates and get
	\begin{equation}
	\label{gdslambda1}
	\begin{cases}
	\|u\|_\infty \leq \e^\beta\frac{2MD^2}{\sin \vartheta}\lp \|U\|_\infty+\|v\|_\infty\rp  \\
	\|v\|_\infty \leq \e^\beta\frac{2M}{\sin \vartheta}\lp \|V\|_\infty+\frac{\lV\mu\rV_\infty}{\e}\lV u\rV_\infty\rp.
	\end{cases}
	\end{equation}
	Using the first equation of (\ref{gdslambda1}) in the second one yields
	$$
	\|v\|_\infty\lp1-\e^{2\beta-1}\lp\frac{2MD}{\sin \vartheta}\rp^2\lV\mu\rV_\infty\rp\leq 
	\e^\beta \frac{2M}{\sin\vartheta}\|V\|_\infty+\e^{2\beta-1}\lV\mu\rV_\infty \lp\frac{2MD}{\sin \vartheta}\rp^2\|U\|_\infty
	$$
	i.e., for $\e$ small enough,
	\begin{equation}
	\label{gdslambda2}
	\|v\|_\infty \leq K_1\max(\e^\beta,\e^{2\beta-1})(\|V\|_\infty+\|U\|_\infty)
	\end{equation}
	with $K_1$ depending only on $D,\mub,\vartheta.$ In the same vein, using (\ref{gdslambda2}) in the second equation 
	of (\ref{gdslambda1}) produces the same estimate, and the proof is concluded.
\end{proof}

\subsection{Small values of $\lb\la\rb$} 
The values are treated with the help of the Fourier transform in $x$-direction.
For $U\in \MC_0(\R)\cap L^1(\R),$ $V\in \MC_0(\R^2)\cap L^1(\R,L^\infty(\R)),$ 
we define 
$$
\Uh(\xi):=\int_{\R}e^{-ix\xi}U(x)dx \qquad \Vh(\xi,y):=\int_{\R}e^{-ix\xi}V(x,y)dx.
$$

\begin{prop}
\label{petitslambda}
There exist $\e_1$ a constant $C_2$ depending only on $D,$ $\mub$ and $\vartheta$ such that for all $\e<\e_1,$ for $\beta>\frac{1}{2},$ $\gamma>0,$ 
such that $1-\frac{3}{2}(\b+\gamma)>0,$ if $\la\in \Gamma_{r,\vartheta} \textrm{ and } |\la|<\e^{-\beta}$ then
$$
\lV\lp R(\la,L_\e)-R(\la,L_0)\rp(U,V)\rV_\infty\leq C_2\e^{1-\frac{3}{2}(\beta+\gamma)}
\lp \Vert\Uh\|_{L^\infty(\R)}+\Vert\Vh\|_{L^\infty(\R^2)}\rp.
$$
\end{prop}

\begin{cor}\label{corl1}
 Under assumptions of Proposition \ref{petitslambda}, 
 $$\lV\lp R(\la,L_\e)-R(\la,L_0)\rp(U,V)\rV_\infty\leq C_2\e^{1-\frac{3}{2}(\beta+\gamma)}
 \lp \Vert U\Vert_{L^1(x)}+\lV \lV V\rV_{L^\infty(y)}\rV_{L^1(x)}\rp.$$
\end{cor}

The proof requires two lemmas. First, we deal with the high frequencies 
in Lemma \ref{gdesfrequences}, i.e. for $|\xi|\gg\e^{-\b},$ using Lemma \ref{katoineq} in Appendix. Then, 
in Lemma \ref{petitesfrequences}, we make
an almost explicit computation of the solutions for small values of $\la.$

\paragraph{Fourier transform of the equations}
Let us consider $U\in \MC_0(\R)\cap L^1(\R)$ and $V\in \MC_0(\R^2)\cap L^1(\R,L^\infty(\R)).$
For $\e>0$ and $\la\in\Gamma_{r,\vartheta},$ $|\la|<\e^{-\beta},$ Recall that
 \begin{equation*} 
 \begin{cases}
 (u,v)=R(\la,L_0)(U,V) \\
 (u_\e,v_\e)=R(\la,L_\e)(U,V)
 \end{cases}
 \end{equation*}
 which leads to the spectral problems (\ref{avtfourier0}) and (\ref{avtfouriereps}). The Fourier transforms $(\uh,\vh)$ and $(\uh_\e,\vh_\e)$
 solve
 \begin{equation}\label{apresfourier0}
 \begin{cases}
 (D\xd+\la+\mub) \uh(\xi) = \vh(\xi,0) + \Uh(\xi) \\
 (\xd+\la)\vh - \partial_{yy}\vh(\xi,y) = \Vh(\xi,y)  \\
 \vh(\xi,0^+)=\vh(\xi,0^-) \\
 -(\partial_y v(\xi,0^+)-\partial_y v(\xi,0^-)) = \mub \uh(\xi)-\vh(\xi,0)
 \end{cases}
 \end{equation}
 and
 \begin{equation}\label{apresfouriereps}
 \begin{cases}
 (D\xd+\la+\mub) \uh_\e(\xi) = \int\nu_\e(y)\vh_\e(\xi,y)dy + \Uh(\xi) \\
 (\xd+\la+\nu_\e(y)) \vh_\e(\xi,y) -\partial_{yy} \vh_\e(\xi,y) = \mu_\e(y)\uh_\e(\xi)+\Vh(\xi,y).
 \end{cases}
 \end{equation}
%

 \begin{lem}\label{gdesfrequences}
There exist $\e_2>0$ and a constant $C_3$ depending only on $\mub,D$ such that for all $\e<\e_2,$ for all $\xi$ with 
$\xd\geq \e^{-\beta-\gamma}$ and $\la\in\Gamma_{r,\vartheta}$ with $|\la|<\e^{-\b},$  
\begin{align*}
|\uh_\e(\xi)|+\|\vh_\e(\xi)\|_\infty\leq \frac{C_3}{\xd}\lp|\Uh(\xi)|+\|\Vh(\xi)\|_\infty\rp \\
|\uh(\xi)|+\|\vh(\xi)\|_\infty\leq \frac{C_3}{\xd}\lp|\Uh(\xi)|
+\|\Vh(\xi)\|_\infty\rp
\end{align*}
where $\|.\|_\infty=\|.\|_{L^\infty(y)}.$
 \end{lem}
 
\begin{proof}
We give the proof only for the nonlocal case, the local one being easier. 
Combining the two equations of (\ref{apresfouriereps}), $\vh_\e$ satisfies
\begin{equation}\label{apresfouriereps2}
 -\partial_{yy}\vh_\e(\xi,y)+\lp\xi^2+\la+\nu_\e(y)\rp\vh_\e(\xi,y)=\Vh(\xi,y)+\frac{\mu_\e(y)}{D\xi^2+\la+\mub}\lp\vh(\xi,0)+\Uh(\xi)\rp.
\end{equation}
As $\gamma>0$ and considering the hypotheses on $\la$ and $\xi,$ there exists $k>0$ such that, for $\e$ small enough, we have:
\begin{equation*}
 \min \left\{ \Re{\xi^2+\la},\lb D\xi^2+\la+\mub\rb\right\} >k^2\xi^2.
\end{equation*}
Now, we apply Lemma \ref{katoineq} in Appendix. It gives:
\begin{equation*} 
 \lb\vh_\e(\xi,y)\rb  \leq \frac{1}{2k|\xi|}\int_\R e^{-k\xi|z|}\lp\Vh(\xi,z-y)+\frac{\mu_\e(z-y)}{k^2\xi^2}\lp \vh(\xi,0)+\Uh(\xi)\rp \rp dz .
\end{equation*}
A rough majoration yields
\begin{equation*}
  \lV\vh_\e(\xi)\rV_\infty  \leq \frac{1}{k^2\xi^2}\lV\Vh(\xi)\rV_\infty+\frac{\mub}{k^3|\xi|^3}\lp\lb\vh(\xi,0)\rb+\lb\Uh(\xi)\rb\rp
\end{equation*}
which, as $|\xi|>\e^{-\frac{1}{2}(\b+\gamma)},$ provides the desired estimate on $\lV\vh_e(\xi)\rV_\infty.$ The estimate on $|\uh_\e|$ follows from
the first equation of (\ref{apresfouriereps}), and the proof of Lemma \ref{gdesfrequences} is concluded.
\end{proof}

Now that we are done with the high frequencies, to finish the proof of Proposition \ref{petitslambda},
it remains to control the lower frequencies. This is the purpose of the following Lemma.
 
\begin{lem}
\label{petitesfrequences}
There exist $\e_3>0$ and a constant $C_5$, for all $\e<\e_3,$ for all $\la\in\Gamma_{r,\vartheta}$ with $|\la|<\e^{-\beta},$ 
for all $\xi$ with $\xd<\e^{-(\beta+\gamma)},$
\begin{align*}
\|\vh_\e(\xi)-\vh(\xi)\|_{L^\infty(y)} & \leq C_5 \e^{1-(\beta+\gamma)}\lp|\Uh(\xi)|+\|\Vh(\xi)\|_{L^\infty(y)}\rp \\
\lb\uh_\e(\xi)-\uh(\xi)\rb & \leq C_5 \e^{1-(\beta+\gamma)}\lp|\Uh(\xi)|+\|\Vh(\xi)\|_{L^\infty(y)}\rp.
\end{align*}
\end{lem} 

\paragraph{Proof of Proposition \ref{petitslambda} thanks to Lemma \ref{petitesfrequences}}
With the same notations,
\begin{equation*}
 \lV\lp R(\la,L_\e)-R(\la,L_0)\rp(U,V)\rV_\infty=\max \lp\lV u-u_\e\rV_{L^\infty(\R)},\lV v-v_\e\rV_{L^\infty(\R^2)}\rp.
\end{equation*}
Let us prove the domination for $\lV v-v_\e\rV_{L^\infty(\R^2)},$ the one in $(u-u_\e)$ being similar. For $(x,y)\in\R^2,$
\begin{align}
 (v-v_\e)(x,y) & = \frac{1}{2\pi}\int_\R e^{ix.\xi}\lp \vh(\xi,y)-\vh_\e(\xi,y)\rp d\xi. \nonumber \\
 \lb (v-v_\e)(x,y)\rb & \leq \frac{1}{2\pi}\int_\R \lb \vh(\xi,y)-\vh_\e(\xi,y)\rb d\xi  \nonumber \\
 & \leq \frac{1}{2\pi} \int_{|\xi|\geq\e^{-\frac{1}{2}(\b+\gamma)}} \lV \vh(\xi)\rV_{L^\infty(y)}+\lV \vh_\e(\xi)\rV_{L^\infty(y)} d\xi \label{controle1} \\
 & + \frac{1}{2\pi} \int_{|\xi|<\e^{-\frac{1}{2}(\b+\gamma)}} \lV \vh(\xi)-\vh_\e(\xi)\rV_{L^\infty(y)} d\xi. \label{controle2}
\end{align}
Now, from Lemma \ref{gdesfrequences}, we have:
\begin{align*}
 (\ref{controle1}) & \leq \frac{C_3}{\pi}\lp |\Uh(\xi)|+\|\Vh(\xi)\|_\infty\rp \int_{\e^{-\frac{1}{2}(\b+\gamma)}}^{+\infty} \frac{d\xi}{\xd} \\
  & \leq  \frac{C_3}{\pi}\e^{\frac{1}{2}(\b+\gamma)}\lp |\Uh(\xi)|+\|\Vh(\xi)\|_\infty\rp.
\end{align*}

From Lemma \ref{petitesfrequences}, we have:
\begin{align*}
 (\ref{controle2}) & \leq \frac{C_5}{2\pi}\e^{1-(\beta+\gamma)}\lp|\Uh(\xi)|+\|\Vh(\xi)\|_{L^\infty(y)}\rp 
 \int_{-\e^{-\frac{1}{2}(\b+\gamma)}}^{\e^{-\frac{1}{2}(\b+\gamma)}} 1 d\xi \\
  & \leq \frac{C_5}{\pi}\e^{1-\frac{3}{2}(\beta+\gamma)}\lp|\Uh(\xi)|+\|\Vh(\xi)\|_{L^\infty(y)}\rp.
\end{align*}
Finally, from the choice of $\b,\gamma,$ $\frac{\b+\gamma}{2}>1-\frac{3}{2}(\b+\gamma)$ and we have the required estimate.
\qed

\subsection*{Proof of Lemma \ref{petitesfrequences}}
The proof requires some explicit computations of the solutions of the spectral problems and is a bit long. First, we compute
$(\uh,\vh),$ solution of (\ref{apresfourier0}), introducing four constants $K_1^+,$ $K_2^+,$ $K_1^-,$ $K_2^-$.
Secondly, we do the same for $(\uh_\e,\vh_\e)$ outside the strip $\R\times(-\e,\e).$
This leads us to introduce four constants $K_1^+(\e),$ $K_2^+(\e),$ $K_1^-(\e),$ and $K_2^-(\e)$ which determine the behaviour 
of $\vh_\e$ outside the strip. Then, we show that it is enough to focus on $K^\pm_2(\e)$ to get a global control of $\vh_\e.$
In a short paragraph, we establish that the derivative $\partial_y \vh_\e$ is controlled by the norm of $\vh_\e.$ The last paragraph of the proof
is devoted to the computation of $K^\pm_2(\e)$ and an estimate of its difference with $K_2^\pm.$

\paragraph{Explicit computation of $(\uh,\vh)$}
	Our choice of $\Gamma_{r,\vartheta}$ allows us to choose a unique determination of the complex logarithm for all systems. From now and until the end of
	this proof we set for all $\xi$ and $\la$ satisfying the hypotheses of Lemma \ref{petitesfrequences}
	\begin{equation}
	\label{alpha}
	\a:=\sqrt{\xd+\la},\ \Re\a>0
	\end{equation}
	the unique complex root of $(\xd+\la)$ with positive real part and 
	\begin{equation}
	\label{omega}
	\o:=D\xd+\mub+\la.
	\end{equation}
	The choice of $\vp,\theta$ (see the hypotheses of Proposition \ref{sectorial})
	yields $\displaystyle \min_{\la\in\Gamma_{r, \vartheta}} |\la|>\max (\sqrt{2},2\mub).$  
	Moreover, we have $\xd>0$ and $D\xd+\mub>0.$ Then, 
	\begin{equation}
	\forall \las\in \Gamma_{r, \vartheta},\ \max (\sqrt{2},2\mub) < \min_{\la\in\Gamma_{r, \vartheta}} |\la| \leq \min (D\xd+\mub+\las,\xd+\las).
	\end{equation}
	Hence, we can assert:
	\begin{equation}\label{controlomega}
	\forall\la\in\Gamma_{r, \vartheta},\xi\in\R, \ 2< 2+\frac{1}{\a}\lp 1-\frac{\mub}{\o}\rp<2+2^{-\frac{1}{4}}+\frac{1}{\sqrt{2\mub}}.
	\end{equation}
	Moreover, considering the hypotheses on $\xi$ and $\la,$ there exists a constant $k>0$ such that, for $\e$ small enough,
	\begin{equation}\label{controlalpha}
	 1<|\a|<k\e^{-\frac{1}{2}(\b+\gamma)},\qquad 2\mub\leq|\o|<k\e^{-(\b+\gamma)},\qquad |e^{\e\a}|<2.
	\end{equation}

	The first equation of (\ref{apresfourier0}) gives 
	$$
	\uh(\xi)=\frac{1}{\o}\lp\vh(\xi,0)+\Uh(\xi)\rp.
	$$
	Integrating the second equation of (\ref{apresfourier0}) yields the existence of four 
	constants $K_1^+,$ $K_2^+,$ $K_1^-,$ $K_2^-$, depending on $\xi,$ such that
	\begin{equation}\label{vh0}
	\left\{
	\begin{array}{ll}
	\vh(\xi,y) =  e^{\a y}\lp K_1^+ -\frac{1}{2\a}\int_0^y e^{-\a z}\Vh(\xi,z)dz\rp +
	e^{-\a y}\lp K_2^+ +\frac{1}{2\a}\int_0^y e^{\a z}\Vh(\xi,z)dz\rp & y>0 \\
	\vh(\xi,y) =  e^{-\a y}\lp K_1^- -\frac{1}{2\a}\int_y^0 e^{\a z}\Vh(\xi,z)dz\rp +
	e^{\a y}\lp K_2^- +\frac{1}{2\a}\int_y^0 e^{-\a z}\Vh(\xi,z)dz\rp & y<0.
	\end{array}
	\right.
	\end{equation}
	The integrability of $\vh$ in $y$ gives
	\begin{equation}
	\label{k10}
	K_1^+=\frac{1}{2\a}\int_0^\infty e^{-\a z}\Vh(\xi,z)dz, \qquad
	K_1^-=\frac{1}{2\a}\int_{-\infty}^0 e^{\a z}\Vh(\xi,z)dz.
	\end{equation}
	The continuity and exchange conditions at $y=0$ impose 
	\begin{equation}
	\begin{cases}
	K_1^+ +K_2^+ = K_1^- +K_2^- \\
	\a \lp K_2^+-K_1^++K_2^--K_1^-\rp = \frac{\mub}{\o}\lp K_1^++K_2^++\Uh(\xi)\rp-K_1^+-K_2^+.
	\end{cases}
	\end{equation}
	Combining these two equations yields
	\begin{equation}\label{k20}
	\begin{cases}
	K_2^+\lp 2+\frac{1}{\a}\lp 1-\frac{\mub}{\o}\rp\rp = 2K_1^-+K_1^+\frac{1}{\a}\lp\frac{\mub}{\o}-1\rp+\frac{\mub}{\a\o}\Uh(\xi) \\
	K_2^-\lp 2+\frac{1}{\a}\lp 1-\frac{\mub}{\o}\rp\rp = 2K_1^++K_1^-\frac{1}{\a}\lp\frac{\mub}{\o}-1\rp+\frac{\mub}{\a\o}\Uh(\xi).
	\end{cases}
	\end{equation}
	From (\ref{controlomega}), the above system (\ref{k20}) is well-posed.
	From (\ref{vh0}), (\ref{k10}) and (\ref{k20}) we have an explicit formula for $(\uh(\xi),\vh(\xi)).$

	\paragraph{Study of $(\uh_\e,\vh_\e)$}
	\subparagraph{Explicit formula} In the same way as above, the first equation of (\ref{apresfouriereps}) yields
	$$
	\uh_\e(\xi)=\frac{1}{\o}\lp\int_{\R}\nu_\e(y)\vh_\e(\xi,y)dy+\Uh(\xi)\rp.
	$$
	Integrating the second equation of (\ref{apresfouriereps}) leads us to set four constants 
	$K_1^+(\e),$ $K_2^+(\e),$ $K_1^-(\e),$ $K_2^-(\e)$, depending on $\e$ and $\xi$,
	such that
	\begin{align}
	y>\e: \ \vh_\e(\xi,y)  = & e^{\a y}\lp K_1^+(\e) -\frac{1}{2\a}\int_\e^y e^{-\a z}\Vh(\xi,z)dz\rp + \nonumber \\\label{expliciteps1}
	& e^{-\a y}\lp K_2^+(\e) +\frac{1}{2\a}\int_\e^y e^{\a z}\Vh(\xi,z)dz\rp  \\
	y<-\e: \ \vh_\e(\xi,y)  = & e^{-\a y}\lp K_1^-(\e) -\frac{1}{2\a}\int_y^{-\e} e^{\a z}\Vh(\xi,z)dz\rp + \nonumber \\\label{expliciteps2}
	& e^{\a y}\lp K_2^-(\e) +\frac{1}{2\a}\int_y^{-\e} e^{-\a z}\Vh(\xi,z)dz\rp .
	\end{align}
	For the same integrability reason as in the limit case, we already have an explicit
	formula for $K_1^\pm(\e):$
	\begin{equation}\label{k1eps}
	K_1^+(\e)=\frac{1}{2\a}\int_\e^\infty e^{-\a z}\Vh(\xi,z)dz, \qquad
	K_1^-(\e)=\frac{1}{2\a}\int_{-\infty}^{-\e} e^{\a z}\Vh(\xi,z)dz,
	\end{equation}
	which immediately gives us a uniform boundary and, combining with (\ref{k10}), the
	first following estimate
	\begin{equation}\label{estimeek1}
	\lb K_1^\pm-K_1^\pm(\e)\rb\leq \frac{\e}{2\a}\lV\Vh(\xi)\rV_{L^\infty(y)}.
	\end{equation}
	It remains to determine $K_2^\pm(\e).$ We set
	\begin{equation}\label{newvariable}
	 z=\frac{y}{\e} \textrm{ and } \vh_\e(z):=\vh_\e(y) \textrm{ for }z\in(-1,1).
	\end{equation}
	The equation for $\vh_\e(\xi,z)$, now set for $z\in(-1,1),$ is
	\begin{equation}\label{eqeps}
	(\e^2\xd+\e^2\la+\e\nu(z)) \vh_\e(\xi,z) -\partial_{zz} \vh_\e(\xi,z) = 
	\e\mu(z)\frac{1}{\o}\lp\int_{\R}\nu(z)\vh_\e(\xi,z)dz+\Uh(\xi)\rp+\e^2\Vh(\xi,\e z).
	\end{equation}
	Specifying (\ref{expliciteps1}) and (\ref{expliciteps2}) at $y=\e$
	and $y=-\e$ gives us the two following boundary conditions for (\ref{eqeps}):
	\begin{equation}\label{dirichleteps}
	\begin{cases}
	\vh_\e(1,\xi) &=K_1^+(\e)e^{\a\e}+K_2^+(\e)e^{-\a\e} \\
	\vh_\e(-1,\xi) &=K_1^-(\e)e^{\a\e}+K_2^-(\e)e^{-\a\e}.
	\end{cases}
	\end{equation}
	\begin{equation}\label{neumaneps}
	\begin{cases}
	\partial_z \vh_\e(1,\xi)& =\e\a\lp K_1^+(\e)e^{\a\e}-K_2^+(\e)e^{-\a\e}\rp \\
	\partial_z \vh_\e(-1,\xi)& =\e\a\lp K_2^-(\e)e^{-\a\e}-K_1^-(\e)e^{\a\e}\rp.
	\end{cases}
	\end{equation}
	
	\subparagraph{Blow-up condition for $\vh_\e$} From now, 
	we are only considering the rescaled equation (\ref{eqeps}) with the boundary conditions
	(\ref{dirichleteps}) and (\ref{neumaneps}). Hence, all functions and derivatives
	are to be considered in these rescaled variables (\ref{newvariable}). 
	We first show that the $L^\infty(z)-$norm of $\vh_\e$ is controlled by $K_2^\pm(\e).$
	We have: 
	\begin{align*}
	\vh_\e(z)-\vh_\e(-1) &= (z+1)\vh_\e'(-1)+\int_{-1}^{z}\int_{-1}^{s}\vh_\e''(u)duds \\
	&= (z+1)\vh_\e'(-1)+\int_{-1}^{z}\int_{-1}^{s}\vh_\e(u)\lp \e\nu(u)+\e^2\xd+\e^2\la\rp duds \\
	& - \int_{-1}^{z}\int_{-1}^{s}\e\frac{\mu(u)}{\o}\lp \int\nu\vh_\e+\Uh(\xi)\rp+\e^2\Vh(\xi,\e u)duds.
	\end{align*}
	Hence
	\begin{align*}
	\|\vh_\e\|_{L^\infty(-1,1)}  \leq & \lb\vh_\e(-1)\rb+2\lb\vh_\e'(-1)\rb+\e\|\vh_\e\|_\infty\lp 2+\frac{2\mub}{|\o|}+4\e\lb\xd+\la\rb\rp \\
	& +\e\frac{2\mub}{|\o|}\lb\Uh(\xi)\rb+4\e^2\|\Vh(\xi)\|_\infty \\
	\leq & \lp|K_1^-(\e)|+|K_2^-(\e)|\rp(2+4|\e\a|)+\e\|\vh_\e\|_\infty\lp 4+4k^2\e^{1-(\b+\gamma)}\rp \\
	& +2\e\lb\Uh(\xi)\rb+4\e^2\|\Vh(\xi)\|_\infty.
	\end{align*}
	Now let us recall that $K_1^-(\e)$ is uniformly bounded in $\e$ from (\ref{estimeek1}), 
	we have $\b+\gamma<1,$ and $\xi\mapsto(\Uh(\xi),\Vh(\xi))$ is uniformly bounded. 
	These facts, combined with the above inequality and the symmetry of the problem, allow us to assert that
	\begin{equation}\label{borneK2}
	\lp\underset{\e\to0}{\lim\sup}\|\vh_\e\|_{L^\infty(y)}=+\infty \rp
	\Leftrightarrow
	\lp\underset{\e\to0}{\lim\sup}\lb K_2^-(\e)\rb=+\infty \rp
	\Leftrightarrow
	\lp\underset{\e\to0}{\lim\sup}\lb K_2^+(\e)\rb=+\infty \rp
	\end{equation}
	 uniformly in $\xi\in\R.$
	
	\subparagraph{Control of the derivative}
	In the same way as above we get a control of $\|\vh_\e'\|_\infty$ by $\|\vh_\e\|_\infty$ with
	a simple integration of (\ref{eqeps}):
	\begin{align*}
	\vh_\e'(z)-\vh_\e'(-1) = & \e\int_{-1}^{z}\vh_\e(s)\lp \nu(s)+\e\xd+\e\la\rp-\frac{\mu(s)}{\o}\lp\int\nu\vh_\e+\Uh(\xi)\rp-\e\Vh(\xi,\e s)ds \\
	\lb \vh_\e'(s)-\vh_\e'(-1)\rb \leq & \e\|\vh_\e\|_\infty\lp 2+k^2\e^{1-(\b+\gamma)})\rp+\e\lp |\Uh(\xi)|+4\e\|\Vh(\xi)\|_\infty\rp.
	\end{align*}
	So for $\e$ small enough,
	\begin{equation}\label{controlderivee}
	\|\vh_\e'\|_{L^\infty(-1,1)}\leq 8\e\lp\|\vh_\e\|_\infty+1\rp+4\e\lp |\Uh(\xi)|+\e\|\Vh(\xi)\|_\infty\rp.
	\end{equation}
	
	\subparagraph{Explicit computation of $K_2^\pm(\e)$}We are now ready to prove
	Lemma \ref{petitesfrequences}. The only argument we will use is, once again, an 
	integration of (\ref{eqeps}). It yields:
	\begin{align*}
	\vh_\e(1)-\vh_\e(-1) = & 2\vh_\e'(-1)+\e\int_{-1}^1\int_{-1}^z\vh_\e(s)\lp\nu(s)+\e\xd+\e\la\rp dsdz \\
	& -\e\int_{-1}^1\int_{-1}^z\frac{\mu(s)}{\o}\lp\int\nu\vh_\e+\Uh(\xi)\rp+\e\Vh(\xi,\e s)dsdz.
	\end{align*}
	Hence, with (\ref{dirichleteps}) and the estimate (\ref{controlalpha}) on $\a$, we have
	\begin{multline}\label{calculk2}
	\Big\vert e^{-\a\e}\lp K_2^+(\e)-K_2^-(\e)\rp - e^{\a\e}\lp K_1^-(\e)-K_1^+(\e)\rp \Big\vert \\
	\leq 2\lb\vh_\e'(-1)\rb +\e\lV\vh_\e\rV_\infty\lp 6+4k^2\e^{1-(\b+\gamma)}\rp 
	     +\e\lp 4\lb\Uh(\xi)\rb+4\e\lV\Vh(\xi)\rV_\infty\rp. 
	\end{multline}
	In the same fashion,
	\begin{equation*}
	\vh_\e'(1)-\vh_\e'(-1)=\e\int_{-1}^1 \vh_\e(z)\lp\nu(z)+\e(\xd+\la)\rp-\frac{\mu(z)}{\o}\lp\int\nu\vh_\e+\Uh\rp-\Vh(\e z) dz. 
	\end{equation*}
	Hence, using (\ref{neumaneps}),
	\begin{multline}\label{calculK21}
	\lb \a e^{\a\e}\lp K_1^+(\e)+K_1^-(\e)\rp -\a e^{-\a\e}\lp K_2^+(\e)+K_2^-(\e)\rp - 
	\lp 1-\frac{\mub}{\o}\rp\int\nu\vh_\e + \frac{\mub}{\o}\Uh(\xi)\rb  \\ 
	\leq 2k^2\e^{1-(\b+\gamma)}\lV\vh_\e\rV_\infty+2\e\lV\Vh(\xi)\rV_\infty. 
	\end{multline}
	Now we just do a Taylor-Lagrange expansion: for all $\e,\xi,\la,$ there exists a function $c:[-1,1]\mapsto[-1,1]$ such that
	\begin{equation}\label{taylornuv}
	\int_{-1}^{1}\nu(z)\vh_\e(z)dz = \lp K_1^+(\e)e^{\a\e}+K_2^+(\e)e^{-\a\e}\rp+\int_{-1}^{1}\vh_\e'(c(z))(z-1)\nu(z)dz.
	\end{equation}
	Using (\ref{taylornuv}) in (\ref{calculK21}) gives
	\begin{multline}\label{calculk22}
	  \Big\vert e^{-\a\e}K_2^+(\e)\lp\a+\lp 1-\frac{\mub}{\o}\rp\rp + \a e^{-\a\e}K_2^-(\e) 
	+ e^{\a\e}K_1^+(\e)\lp\lp 1-\frac{\mub}{\o}\rp-\a\rp \\ - \a e^{\a\e}K_1^-(\e)  
	 - \frac{\mub}{\o}\Uh(\xi) \Big\vert  \leq 2\lV\vh'_\e\rV_\infty + 2k^2\e^{1-(\b+\gamma)}\lV\vh_\e\rV_\infty+2\e\lV\Vh(\xi)\rV_\infty.
	\end{multline}
	At this point, we have a system of two inequations (\ref{calculk2}) and (\ref{calculk22}) which will allow us to compute 
	an approximation of $K_2^+(\e)$ and $K_2^-(\e).$ We will give details only for $K_2^+(\e),$ the other case being similar. 
	Let us consider $e^{\a\e}\left[(\ref{calculk2})-\frac{1}{\a}(\ref{calculk22})\right],$ still using (\ref{controlalpha}). 
	This reads:
	\begin{multline*}
	\lb K_2^+(\e)\lp 2+\frac{1}{\a}\lp 1-\frac{\mub}{\o}\rp\rp - 2e^{2\a\e}K_1^-(\e) - e^{2\a\e}K_1^+(\e)\frac{1}{\a}\lp\frac{\mub}{\o}-1\rp
	- e^{\a\e}\frac{\mub}{\a\o}\Uh(\xi) \rb  \\ \leq
	 4\lb\vh_\e'(-1)\rb +\e\lV\vh_\e\rV_\infty\lp 6+4k^2\e^{1-(\b+\gamma)}\rp  
	     +2\e\lp 4\lb\Uh(\xi)\rb+4\e\lV\Vh(\xi)\rV_\infty\rp + \\
	     4\lV\vh'_\e\rV_\infty + 4k^2\e^{1-(\b+\gamma)}\lV\vh_\e\rV_\infty+4\e\lV\Vh(\xi)\rV_\infty.
	\end{multline*}
	Let us recall that 
	(\ref{controlderivee}) gives us a control of $\|\vh_\e'\|_{L^\infty(z)}$ 
	by $\e\|\vh\|_{L^\infty(z)}$ in the strip $[-1,1].$ 
	Thus, for some constant $C_6,$ we have:
	\begin{multline}\label{calculk2fin}
	 \lb K_2^+(\e)\lp 2+\frac{1}{\a}\lp 1-\frac{\mub}{\o}\rp\rp - 2K_1^-(\e) - K_1^+(\e)\frac{1}{\a}\lp\frac{\mub}{\o}-1\rp
	- \frac{\mub}{\a\o}\Uh(\xi) \rb  \\ \leq
	C_6\lp |\e\a|\lp |K_1^+(\e)|+|K_1^-(\e)|\rp+\e^{1-(\b+\gamma)}\lV\vh_\e\rV_\infty+\e\lb\Uh(\xi)\rb+\e\lV\Vh(\xi)\rV_\infty\rp.
	\end{multline}

	The last expression (\ref{calculk2fin}) 
	combined with the control of $\lV\vh_\e\rV_\infty$ by $K_2^+(\e)$ given in (\ref{borneK2}) allows us to assert that
	$\lp\|z\mapsto \vh_\e(\xi,z)\|_{L^\infty(-1,1)}\rp_\e$ is uniformly bounded on $\e,\xi,\la$ 
	under assumptions of Lemma \ref{petitesfrequences}, and so is 
	$\lp\|y\mapsto \vh_\e(\xi,y)\|_{L^\infty(\R)}\rp_\e$ with (\ref{expliciteps1}) and 
	(\ref{expliciteps2}). Comparing (\ref{calculk2fin}) with (\ref{k20}) and using the previous estimate 
	(\ref{estimeek1}) and the explicit formula for $K_1^\pm(\e)$ yields, for some constant $C_7:$
	\begin{equation}\label{estimeek2}
	\lb K_2^+(\e)-K_2^+\rb\leq C_7\e^{1-\b-\gamma}\lp\|\vh(\xi)\|_{L^\infty(y)}+\lb\Uh(\xi)\rb
	+\|\Vh(\xi)\|_{L^\infty(y)}\rp.
	\end{equation}
	Now we are done with the rescaled variables.
	To conclude the proof of Lemma \ref{petitesfrequences}, 
	we compute directly the difference from (\ref{vh0}) and (\ref{expliciteps1}), (\ref{expliciteps2}). 
	We have the explicit formulas (\ref{k10}) and (\ref{k1eps}), the other terms being treated by (\ref{estimeek2}). 
	As for our previous estimate (\ref{estimeek2}), we will only focus 
	on the case $y>0,$ the other one $y<0$ being similar. All in all, we have
	\begin{align*}
	\lb\vh_\e(\xi,y)-\vh(\xi,y)\rb & \leq e^{-\a y}\lb K_2^+(\e)-K_2^+\rb+
	\frac{1}{|2\a|}\int_0^{\e}\lb e^{\a(z-y)}\Vh(\xi,z)\rb dz \\
	& \leq C_7\e^{1-\b-\gamma}\lp\|\vh(\xi)\|_\infty+|\Uh(\xi)|
	+2\|\Vh(\xi)\|_\infty\rp
	\end{align*}
	 and the proof of Lemma \ref{petitesfrequences} is finished. 
\qed

\section{Finite time convergence}\label{sectionconvsol}
In this section, we finish the proof of Theorem \ref{convsol}. The first ingredient
is Proposition \ref{cordif} which is a corollary of our estimates on the resolvents. It gives a control
of the semigroups generated by the linear operators. The second ingredient is 
Lemma \ref{decroissancesolutions} which gives some $L^\infty$ and decay estimates on the solutions.
At last, we will use these two ingredients in a Gronwall argument to deal with the nonlinearity.

\subsection{Difference between the semigroups}
Combined with the explicit formula (\ref{semigroup}), the results
given in Lemma \ref{grandslambda} and Corollary \ref{corl1} allow us to assert 
the following estimate on the difference between the analytic semigroups.
\begin{prop}\label{cordif}
	Let $t\in (0,T),$ $\b>\frac{1}{2}$ and $\gamma>0$ such that $1-\frac{3}{2}(\b+\gamma)>0.$ There exists a constant $C_8$ depending only on
	$r, \vartheta,D,\mub$ such that, for $\e$ small enough,
\begin{align*}
	\lV\lp e^{tL_0}-e^{tL_\e}\rp\lp U,V\rp\rV_\infty \leq & 
	\e^{1-\frac{3}{2}(\b+\gamma)}\lp\Vert U\Vert_{L^1(x)}+\lV \lV V\rV_{L^\infty(y)}\rV_{L^1(x)}\rp C_8\lp e^{TC_8}+\frac{1}{t}\rp \\
	& +	\e^{2\b-1}\lV\lp U,V\rp\rV_\infty \frac{C_8}{t}e^{-\frac{t}{C_8\e^\b}}.
\end{align*}
\end{prop}
\begin{proof}
\begin{align}
 \lV\lp e^{tL_0}-e^{tL_\e}\rp\lp U,V\rp\rV_\infty \leq & \frac{1}{2\pi}\int_{\Gamma}|e^{t\la}|. 
    \lV\lp R(\la,L_0)-R(\la,L_\e)\rp\lp U,V\rp\rV_\infty d\la \nonumber \\
  \leq & \frac{1}{2\pi}\int_{\la\in\Gamma,|\la|>\e^{-\b}}|e^{t\la}|.\lp\lV R(\la,L_0)\rV+\lV R(\la,L_\e)\rV\rp \lV(U,V)\rV_\infty d\la \label{ligne2} \\
   & + \frac{1}{2\pi}\int_{\la\in\Gamma,|\la|<\e^{-\b}}|e^{t\la}|.\lV\lp R(\la,L_0)-R(\la,L_\e)\rp\lp(U,V)\rp\rV_\infty d\la. \label{ligne3}
\end{align}
We recall that $\frac{\pi}{2}<\vartheta<\frac{3\pi}{4}.$ Hence, for large $\la,$ the curve $\Gamma_{r,\vartheta}$ 
lies in the half-plane $\left\{ z,\Re z <0\right\}.$
From Lemma \ref{grandslambda} (and Proposition \ref{sectorial} for $L_0$), the 
first term of the right handside of the above inequality satisfies, for some constant $C,$
$$
(\ref{ligne2}) \leq 2\int_{s>\e^{-\b}} C e^{-ts}\e^{2\b-1}\lV\lp U,V\rp\rV_\infty ds.
$$
The second term satisfies from Corollary \ref{corl1}
$$
(\ref{ligne3}) \leq \e^{1-\frac{3}{2}(\beta+\gamma)}
 \lp \Vert U\Vert_{L^1(x)}+\lV \lV V\rV_{L^\infty(y)}\rV_{L^1(x)}\rp \int_{\la\in\Gamma,|\la|<\e^{-\b}}|C e^{t\la}|d\la.
$$
It remains to notice that $\Gamma\cap \{z,\ \Re z\geq0\}$ is bounded, and the proof of Proposition \ref{cordif} is complete.
\end{proof}

Remark that $e^{tL_\e},e^{tL_0}\to Id$ as $t\to0.$ So the estimate given in Proposition \ref{cordif} is far from optimal, especially for small $t.$
But it will be enough for our purpose.

\subsection{Uniform decay in $x$}
\begin{lem}\label{decroissancesolutions}
Let $(u,v)(t)$ be the solution of (\ref{BRReq2}), and, for all $\e\in(0,1),$ $(u_\e,v_\e)(t)$ the solution 
of (\ref{RPeps}), both with initial datum $(u_0,v_0).$ Then, there exists $K_2,\las,\cs>0$ independent of $\e$
such that for all 
$(x,y)\in\R^2,t>0,$
$$
\max\ \lp u(t,x),u_\e(t,x),v(t,x,y),v_\e(t,x,y)\rp \leq K_2 e^{-\las\lp |x|-\cs t\rp}.
$$
\end{lem}

\begin{proof}
 \textit{First case: $D>2d.$} Let $\cs$ be the unique positive solution of 
 the equation in $c$ 
$$\la_2^-(c):=\frac{c-\sqrt{c^2-c_{KPP}^2}}{2d} = \frac{c}{D}.$$
To this velocity $\cs$ we associate the
decay rate $\las:=\frac{\cs}{D}.$ Thus, from simple geometric considerations (see Figure \ref{2graphs}),
we have for all $\e>0$ that $c^*_\e,c^*_0<\cs$ and $\la^*_\e,\la^*_0>\las.$
Hence, for all $t>0,x\in\R,\e>0,$
$$
e^{-\la^*_\e(|x|-c^*_\e t)}<e^{-\las(|x|-\cs t)}.
$$
We recall that the linear travelling waves $(c^*_\e,\la_\e^*,\phi_\e^*)$ and $(c^*_0,\la^*_0,\phi^*)$ are
supersolutions for (\ref{RPeps}) and (\ref{BRReq2}).
From Lemma \ref{convergencephi}, 
as $(u_0,v_0)$ is continuous and compactly supported, we know that there exists 
a constant $K_1$ such that
$$
\begin{cases}
u_0(x)\leq K_1 e^{-\las |x|} & \forall x\in\R\\
v_0(x,y) \leq K_1  e^{-\las |x|} \min\{\phi^*(y), \underset{\e\in(0,1]}{\inf}\phi_\e^*(y)\} & \forall (x,y)\in\R^2.
\end{cases}
$$
From Lemma \ref{borneuniformephi}, there exists a constant $K_2$ such that
$$
\underset{\e\geq0}{\sup}\ \underset{y\in\R}{\sup}\ \phi_\e(y)\leq K_2.
$$
up to replace $K_2$ by $\max(K_1,K_2),$ the proof is completed.

\textit{Second case: $D\leq2d.$} In this case, for all $\e>0,$ $c^*_0=c^*_\e=c_{KPP}.$ Let us choose $\cs>c_{KPP}$ and
$\las\in\lp\la_2^-(\cs),\frac{\cs}{D}\rp$ and we conclude in the same fashion.
\end{proof}

\subsection{Proof of Theorem \ref{convsol}}
For $(u,v)\in X,$ set $F(u,v):=\lp 0,f(v)\rp$ the nonlinear term in the studied systems. 
From the regularity of $F$ and Proposition \ref{sectorial}, the solution of (\ref{BRReq2}) $(u,v)$ and of (\ref{RPeps}) $(u_\e,v_\e)$
can be written in the form
\begin{align*}
	(u,v)(t) & = e^{tL_0}(u_0,v_0) + \int_0^t e^{(t-s)L_0}F\lp u(s),v(s)\rp ds \\
	(u_\e,v_\e)(t) & = e^{tL_\e}(u_0,v_0) + \int_0^t e^{(t-s)L_\e}F\lp u_\e(s),v_\e(s)\rp ds.
\end{align*} 

Set $0<\tau_1<T.$ For all $t\in(\tau_1,T),$ 
\begin{align*}
	(u,v)(t)-(u_\e,v_\e)(t) = & \lp e^{tL_0}-e^{tL_\e}\rp(u_0,v_0) \\
	& +	\int_0^t e^{(t-s)L_0}F\lp u(s),v(s)\rp-e^{(t-s)L_\e}F\lp u_\e(s),v_\e(s)\rp ds\\
	= & \lp e^{tL_0}-e^{tL_\e}\rp(u_0,v_0) \\
	& + \int_0^t \lp e^{(t-s)L_0}-e^{(t-s)L_\e}\rp F\lp u_\e(s),v_\e(s)\rp ds \\
	& + \int_0^t e^{(t-s)L_0}\lp F\lp u(s),v(s)\rp-F\lp u_\e(s),v_\e(s)\rp\rp ds.
\end{align*}

It is easy to see (cf \cite{BRR1}) that for all $t>0,$ $\displaystyle \lV e^{tL_0}\rV\leq \max(1,\frac{1}{\mub}).$
Set
$$
\d(t):=\lV(u,v)(t)-(u_\e,v_\e)(t)\rV_\infty,\ \a=\min(1-\frac{3}{2}(\b+\gamma),2\b-1)>0.
$$
and $\d(t)$ satisfies the following inequation:
\begin{align}
\d(t) \leq & \lV\lp e^{tL_0}-e^{tL_\e}\rp\lp u_0,v_0\rp\rV_\infty \label{gron1}\\
 & + \int_0^t \lV \lp e^{(t-s)L_0}-e^{(t-s)L_\e}\rp F\lp u_\e(s),v_\e(s)\rp \rV_\infty ds \label{gron2} \\
 & + \int_0^{\tau_1} \max\lp 1,\frac{1}{\mub}\rp\lV F\rV_{Lip} \lp\lV(u,v)(s)\rV_\infty+\lV(u_\e,v_\e)(s)\rV_\infty\rp ds \label{gron2bis} \\
 & + \int_{\tau_1}^t \max\lp 1,\frac{1}{\mub}\rp\lV F\rV_{Lip} \d(s) ds \label{gron3}
\end{align}
From Proposition \ref{cordif}, we can assert that for some constant $C_9$ depending 
only on $r,$ $\vartheta,$ $D,$ $\mub,$ $\tau_1,$ and $T,$ the first term (\ref{gron1})
satisfies
\begin{equation}\label{gron4}
 \lV\lp e^{tL_0}-e^{tL_\e}\rp\lp u_0,v_0\rp\rV_\infty \leq
 \e^\a C_9 \lp\lp\|u_0\|_\infty+\|v_0\|_\infty\rp\lp 1+|supp(v_0)|+|supp(u_0)|\rp\rp.
\end{equation}

From Lemma \ref{decroissancesolutions}, we get that for all $t>0,\e>0,$
\begin{equation}\label{dominationl1}
\|u_\e(t)\|_\infty,\|v_\e(t)\|_\infty \leq K_2 e^{\las\cs t},\qquad
\lV \lV v_\e(t)\rV_{L^\infty(y)}\rV_{L^1(x)} \leq 2\frac{K_2}{\las}e^{\las\cs t},
\end{equation}
and the same estimates holds for $(u,v).$ So we get for the third term (\ref{gron2bis})
\begin{equation}\label{gron2terce}
\int_0^{\tau_1} \lp\lV(u,v)(s)\rV_\infty+\lV(u_\e,v_\e)(s)\rV_\infty\rp ds \leq
4\tau_1  K_2 e^{\las\cs \tau_1}.
\end{equation}
Recall that $f'(0)>0.$ Hence, any supersolution of (\ref{RPeps}) is also a supersolution of the linear system (\ref{RPepsli}). 
Thus, Lemma \ref{decroissancesolutions} is also available for time-dependent solutions of the linear 
system (\ref{RPepsli}), which in particular entails
\begin{equation}\label{normesemigr}
\forall\e>0,\ \forall t>0,\ \lV e^{tL_\e}\rV \leq K_2 e^{\las\cs t}.
\end{equation}
Now we can deal with the second term (\ref{gron2}). Let us choose $\tau_2>0$ small enough,
and set
\begin{align*}
	\int_0^t \lV \lp e^{(t-s)L_0}-e^{(t-s)L_\e}\rp F\lp u_\e(s),v_\e(s)\rp \rV_\infty ds & \leq \\
		\int_0^{t-\tau_2} \lV \lp e^{(t-s)L_0}-e^{(t-s)L_\e}\rp F\lp u_\e(s),v_\e(s)\rp \rV_\infty ds & + \\
		\int_{t-\tau_2}^t \lp\lV e^{(t-s)L_0}\rV+\lV e^{(t-s)L_\e}\rV\rp \|F\|_{Lip} \lV\lp u_\e(s),v_\e(s)\rp \rV_\infty ds.
\end{align*}
From (\ref{dominationl1}) and (\ref{normesemigr}), 
\begin{align}
\int_{t-\tau_2}^t \lp\lV e^{(t-s)L_0}\rV+\lV e^{(t-s)L_\e}\rV\rp \|F\|_{Lip} \lV\lp u_\e(s),v_\e(s)\rp \rV_\infty ds  & \leq  \nonumber\\
 \tau_2 \|F\|_{Lip}\lp  K_2 e^{\las\cs T}+1+\frac{1}{\mub}\rp\|(u_0,v_0)\|_\infty.
 \label{gron5}
\end{align}
From Proposition \ref{cordif} and (\ref{dominationl1}), 
\begin{align}
	\int_0^{t-\tau_2} \lV \lp e^{(t-s)L_0}-e^{(t-s)L_\e}\rp F\lp u_\e(s),v_\e(s)\rp \rV_\infty ds & \leq \nonumber \\
	\e^\a TC_8\lp e^{TC_8}+\frac{1}{\tau_2}\rp\lp 2+\frac{4}{\las}\rp\lp K_2e^{\las\cs T}\rp. \label{gron6}
\end{align}
We can now conclude the proof of Theorem \ref{convsol} by a classical Gronwall argument in (\ref{gron1})-(\ref{gron3}),
choosing $\tau_1,$ then $\tau_2$ and at last $\e$. Let $\eta>0$ be any small quantity. Let $\tau_1>0$ small enough
such that $\displaystyle 4\tau_1  K_2 e^{\las\cs \tau_1}\leq \frac{\eta}{4}.$ Let $\tau_2>0$ such that 
$$
\tau_2 \|F\|_{Lip}\lp  K_2 e^{\las\cs T}+1+\frac{1}{\mub}\rp\|(u_0,v_0)\|_\infty \leq \frac{\eta}{4}.
$$
Now, let us choose $\e>0$ such that 
\begin{equation*}
\begin{cases}
 \e^\a C_9 \lp\lp\|u_0\|_\infty+\|v_0\|_\infty\rp\lp 1+|supp(v_0)|+|supp(u_0)|\rp\rp \leq \frac{\eta}{4} \\
 \e^\a TC_8\lp e^{TC_8}+\frac{1}{\tau_2}\rp\lp 2+\frac{4}{\las}\rp\lp K_2e^{\las\cs T}\rp \leq \frac{\eta}{4}.
\end{cases}
 \end{equation*}
 Hence, $(\ref{gron1})+(\ref{gron2})+(\ref{gron2bis})\leq \eta$
 and we get from Gronwall's inequality for all $t\in[\tau_1,T]$:
 $$
 \d(t)\leq \eta e^{\|F\|_{Lip}\lp1+\frac{1}{\mub}\rp(T-\tau_1)}.
 $$
\qed

\section{Uniform spreading}

Once again, we consider nonnegative compactly supported initial datum $(u_0,v_0).$
We also made the following assumption on $(u_0,v_0):$
\begin{equation}
 (u_0,v_0)\leq(\frac{1}{\mub}m,m)
\end{equation}
where $m$ is given by (\ref{uniformsupersol}). The purpose is
now to prove Theorem \ref{uniformspreading}.
 Our notations are these of
Section \ref{sectionvitesse}. $c^*_0$ is the asymptotic speed of spreading associated to the limit system (\ref{BRReq2}), 
$c^*_\e$ the one associated to (\ref{RPeps}).
Recall that in the case $D\leq 2d,$ the spreading is driven by the field,
and $c^*_0=c_\e^*=2\sqrt{df'(0)}.$
In both systems, the spreading in this case is independent 
of the line, and the uniform spreading is easy to get.
We will focus on the case $D>2d,$ where the spreading is enhanced by the line.

\paragraph{First part: $c>c^*_0.$}
This is the easiest case. Let $c_1=\frac{c+c^*_0}{2}.$ From Proposition \ref{convergencec}, there exists $\e_0$ such that $\forall\e<\e_0,$ 
$c_\e^*\leq c_1.$ From Lemma \ref{decroissancesolutions}, there 
exists $K$ such that 
$$
\begin{cases}
u_0(x)\leq K_1 e^{-\las |x|} & \forall x\in\R\\
v_0(x,y) \leq K_1  e^{-\las |x|} \underset{\e\in[0,1]}{\inf}\phi_\e(y;c_1) & \forall (x,y)\in\R^2.
\end{cases}
$$
Then, from Proposition \ref{defspreadingspeed} and Lemma \ref{decroissancesolutions},
$$
\forall \e<\e_0,\ u_\e(t,x)\leq K e^{-\las(x-c_1t)},
$$
which concludes the proof of the first part of Theorem \ref{uniformspreading}.
\qed
\paragraph{Second part: $c<c^*_0$}
\subparagraph{Background on subsolutions}
Let us recall that in \cite{BRR1} (resp. in \cite{Pauthier}) the argument to prove the spreading was to devise stationary
compactly supported subsolutions of (\ref{BRReq2}) (resp. (\ref{RPeps})) in a moving framework at some speed $c$ less than
and close to $c^*_0$ (resp. $c^*_\e$). More precisely, for $L$ large enough, set $\Omega^L:=\R\times (-L,L)$ 
and let us consider the following systems for some $\d\ll1$ 
:
\begin{equation}\label{subsolBRR}
 \begin{cases}
 -DU''+cU' = V(x,0)-\mub U &  x\in\R \\
 -d\Delta V+c\partial_x V = \lp f'(0)-\d\rp V & (x,y)\in\Omega^L \\
 -d\lp\partial_x V(x,0^+)-\partial_x V(x,0^-)\rp = \mub U(x)-V(x,0) & x\in\R \\
 V(x,\pm L) = 0 & x\in\R.
 \end{cases}
\end{equation}

\begin{equation}
 \label{subsolRPeps}
\begin{cases}
 -DU''+cU'=-\mub U+\int_{(-L,L)}\nu_\e(y)V(x,y)dy & x\in \R \\
 -d\Delta V+c\partial_x V=\lp f'(0)-\d\rp V+\mu_\e(y)U(x)-\nu_\e(y)V(x,y) & (x,y)\in \Omega^L \\
 V(x,\pm L) = 0 & x\in \R.
\end{cases}
\end{equation}
It was showed by an explicit computation in \cite{BRR1} and an analysis of a spectral problem
in \cite{Pauthier} that there exists a unique $c:=c^*_0(L)$ (resp. $c^*_\e(L)$) such that (\ref{subsolBRR})
(resp. (\ref{subsolRPeps})) admits a unique solution of the form
\begin{equation}\label{soussolexp} 
 \begin{pmatrix}
 U(x) \\
V(x,y)
\end{pmatrix}
= e^{\la x} \begin{pmatrix}
                  1 \\ \vp(y)
                 \end{pmatrix}
\end{equation}
with $c^*_0(L)<c^*_0,$ $c^*_\e(L)<c^*_\e,$ and 
$$
\begin{cases}
 \lim_{\d\to0}\lim_{L\to\infty}c^*_0(L) = \lim_{L\to\infty}\lim_{\d\to0}c^*_0(L) = c^*_0 \\
 \lim_{\d\to0}\lim_{L\to\infty}c^*_\e(L) = \lim_{L\to\infty}\lim_{\d\to0}c^*_\e(L) = c^*_\e.
\end{cases}
$$
Using (\ref{soussolexp}), the system (\ref{subsolRPeps}) reads on
$(c,\la,\vp)$ 
\begin{equation}
\label{eqsoussol}
\begin{cases}
-D\la^2+\la c+\mub  = \int_{(-L,L)} \nu(y)\phi(y)dy \\
-d\vp''(y)+(c\la-d\la^2-f'(0)+\d+\nu_\e(y))\vp(y)  =  \mu_\e(y) & \vp(\pm L) = 0
\end{cases}
\end{equation}
and $c^*_\e(L)$ was given as the first $c$ such that the graphs of the two following 
functions intersect
$$
\begin{cases}
\Psi_1 \ : \ \la \mapsto -D\la^2+\la c+\mub \\
\Psi_2 \ : \ \la \mapsto \int_{(-L,L)}\nu_\e(y)\vp(y)dy
\end{cases}
$$
where, for $\Psi_2,$ $\vp$ is given by the unique solution of second equation
of (\ref{eqsoussol}). Let us call $\Gamma_1,$ resp. $\Gamma_2,$ the graph of $\Psi_1,$ resp. $\Psi_2.$
So we should keep in mind that in (\ref{soussolexp}), both $\la$ and $\vp$ depend on $L,\d,\e.$
Using the same kind of arguments as for Lemma \ref{borneuniformephi} and \ref{convergencephi}, we can 
assert that this dependence is continuous for the $L^\infty$-topology. In particular, the subsolution
(\ref{soussolexp}) of (\ref{subsolRPeps}) converges uniformly in $\d,L$ to the subsolution
of (\ref{subsolBRR}) as $\e$ goes to 0, and of course 
$c^*_\e(L)\to c^*_0(L)$ as $\e\to0$. Hence, the notations are not confusing, as we can continuously extend 
$\vp(\e,\d,L)$ to $\vp(0,\d,L)$ as $\e$ goes to 0.

So we get that both $\vp$ and $\Psi_2$ are:
\begin{itemize}
	\item analytical in $\la,c,\d;$
	\item uniformly continuous in $L$ and $\e,$ up to $\e=0.$
\end{itemize}
Then, a perturbative argument gives for some $c$ less than but close to $c^*(L)$ 
a compactly supported subsolution of (\ref{subsolRPeps}), or (\ref{subsolBRR}) in the limit case $\e=0.$

\subparagraph{Spreading} Let $c<c^*_0.$ Let $c_1=\frac{c+c_0^*}{2},$ $c_2=\frac{c+c_0^*}{4}.$
From Proposition \ref{convergencec}, there exists $\e_1$ such that for all $\e<\e_1,$ 
$c^*_\e>c_2.$ Now, for some $\d$ small enough and some $L$ large enough, \cite{BRR1} and 
\cite{Pauthier} give us a family of subsolutions of (\ref{subsolBRR}) and (\ref{subsolRPeps})
denoted $(\su,\sv)$ and $(\su_\e,\sv_\e)$ for some $c_\e>c_1$ for $\e<\e_1.$ The uniform continuity
of $\Psi_2$ allows us to take same $\d$ and $L$ for all $\e\in[0,\e_1).$
Hence, the convergence result given in Lemma \ref{convergencephi} adapted to this case
gives that $$(\su_\e,\sv_\e)\underset{\e\to0}{\longrightarrow}(\su,\sv)$$
for the $L^\infty$-norm. Set:
$$
\begin{cases}
u_1(x)=\inf\{\su_\e(x),\ \e\in[0,\e_1)\},\ v_1(x)=\inf\{\sv_\e(x,y),\ \e\in[0,\e_1)\} \\
u_2(x)=\sup\{\su_\e(x),\ \e\in[0,\e_1)\},\ v_2(x)=\sup\{\sv_\e(x,y),\ \e\in[0,\e_1)\}.
\end{cases}
$$
Both $u_1,v_1,u_2,v_2$ are nonnegative, continuous and compactly supported. Let us set $\gamma$
such that
$$
\gamma(f'(0)-\d)\lV v_2\rV_\infty \leq f\lp\lV v_2\rV_\infty\rp \textrm{ and }\gamma \lV u_2\rV_\infty <\frac{1}{\mub}. 
$$
We know that $(u,v)(t)$ converges locally uniformly to the steady state $(\frac{1}{\mub},1).$ 
So, let $t_1$ such that $(u,v)(t_1)>\gamma(u_2,v_2).$ From Theorem \ref{convsol}, there exists
$\e_2$ such that for all $\e<\e_2,$ $(u_\e,v_\e)(t_1)>\gamma(u_2,v_2).$ Up to replace 
$\e_2$ by $\min(\e_1,\e_2),$ we get from comparison principle that
$$
\forall t>t_1, \ \forall \e<\e_2,\ \lp u_\e(t,x),v_\e(t,x,y)\rp > \gamma \lp u_1(x-c_1(t-t_1)),v_1(x-c_1(t-t_1),y)\rp.
$$
Let $(\ut,\vt)$ and $(\ut_\e,\vt_\e)$ the solutions of (\ref{BRReq2}) and (\ref{RPeps}) starting at
$t=t_1$ from $\gamma(u_1,v_1).$ Then, considering the hypotheses on $(u_0,v_0),$ for all 
$t>t_1,$ we have:
$$
\begin{cases}
(\ut,\vt)(t)<(u,v)(t)<\lp\frac{1}{\mub},1\rp \\
(\ut_\e,\vt_\e)(t)<(u_\e,v_\e)(t)<\lp U_\e,V_\e\rp.
\end{cases}
$$
Now, from Proposition \ref{convVeps}, there exists $\e_3$ such that if
$\e<\e_3,$ $\displaystyle\lb U_\e-\frac{1}{\mub}\rb <\frac{\eta}{3}.$
Let $t_2$ such that 
$$
\forall x\in \textrm{ supp}(u_1), \forall t>t_2,
\ \lb\ut(t,x)-\frac{1}{\mub}\rb<\frac{\eta}{3}.$$
From Theorem \ref{convsol}, there exists $\e_4$ such that if
$
\e<\e_4,$ $\displaystyle\lb (\ut_\e-\ut)(t_2)\rb<\frac{\eta}{3}.
$
Now, set $\e_0=\min(\e_1,\e_2,\e_3,\e_4),$ $\displaystyle T_0=t_2\frac{c_0^*}{c_0^*-c_2},$ and the
proof of Theorem \ref{uniformspreading} is concluded.
\qed

\newpage

\setcounter{equation}{0}
\begin{appendices}
\renewcommand{\theequation}{A.\arabic{equation}}

\section{}
Here we prove for the convenience of the reader the lemma used in the proof of Lemma \ref{gdesfrequences}. 
It relies on a Kato-type inequality.

\begin{lemA}\label{katoineq}
 Let $\phi,m,f$ be functions in $BUC(\R,\C)$ such that 
 \begin{equation}\label{kato1}
  -\phi''(y)+m(y)\phi(y)=f(y), \qquad y\in\R.
 \end{equation}
 If there exists $\k>0$ such that for all $y\in\R,$ $\Re{m(y)}\geq \k^2$ then 
$|\phi|$ satisfies
\begin{equation}\label{conv}
 |\phi(y)|\leq \frac{1}{2\k}\int_\R e^{-\k|z|}\lb f(z)\rb dz.
\end{equation}
\end{lemA}

\begin{proof}
  Let us first compute the second derivative of the modulus of a complex valued function. Let $\phi\in BUC(\R,\C)\cap \MC^2.$
  An easy computation yields
  $$
  |\phi|''=\frac{\Re{\phib\phi''}}{|\phi|}+\frac{|\phi'|^2}{|\phi|}-\frac{\Re{\phib\phi'}^2}{|\phi|^3}.
  $$
  Hence, for all smooth enough complex-valued function of the real variable, we get
  \begin{equation}\label{modulephiseconde}
   -\lb\phi\rb''\leq-\frac{\Re{\phib\phi''}}{|\phi|}.
  \end{equation}
Now, let us multiply (\ref{kato1}) by $\displaystyle \frac{\phib}{|\phi|}$ and take the real part. It gives
\begin{equation*}
 -\frac{\Re{\phib\phi''}}{|\phi|}+\Re{m}|\phi|=\Re{f\frac{\phib}{|\phi|}}.
\end{equation*}
Using (\ref{modulephiseconde}) in the above inequality yields the following inequation for $|\phi|:$
\begin{equation*}
-|\phi(y)|''+\Re{m(y)}|\phi(y)|\leq|f(y)|. 
\end{equation*}
Now we are reduced to an inequation with real functions. If $\vp$ is the unique solution in $H^1(\R)$ of 
$$
-\vp''(y)+\k^2\vp(y)=\lb f(y)\rb,
$$
from the elliptic maximum principle, we get $|\phi|\leq\vp,$ which is exactly the desired inequality (\ref{conv}). 
\end{proof}

\end{appendices}

 \newpage
\bibliographystyle{plain}
\footnotesize
\bibliography{biblio}

\begin{thebibliography}{10}

\bibitem{AW}
D.~G. Aronson and H.~F. Weinberger.
\newblock Multidimensional nonlinear diffusion arising in population genetics.
\newblock {\em Adv. Math.}, 30:33--76, 1978.

\bibitem{BRRC}
H.~Berestycki, A.-C. Coulon, J.-M. Roquejoffre, and L.~Rossi.
\newblock Exponentially fast fisher-{KPP} propagation induced by a line of
  integral diffusion.
\newblock {\em S\'{e}minaire Laurent Schwartz}, 2014.

\bibitem{BRRC2}
H.~Berestycki, A.-C. Coulon, J.-M. Roquejoffre, and L.~Rossi.
\newblock The effect of a line with non-local diffusion on fisher-{KPP}
  propagation.
\newblock {\em M3AS}, 2015.

\bibitem{BHbook}
H.~Berestycki and F.~Hamel.
\newblock {\em Reaction-Diffusion Equations and Propagation Phenomena}.
\newblock Springer-Verlag, 2014.

\bibitem{BHN1}
H.~Berestycki, F.~Hamel, and N.~Nadirashvili.
\newblock The speed of propagation for {KPP} type problems. {I}. {P}eriodic
  framework.
\newblock {\em J. Eur. Math. Soc. (JEMS)}, 7(2):173--213, 2005.

\bibitem{BHN2}
H.~Berestycki, F.~Hamel, and N.~Nadirashvili.
\newblock The speed of propagation for {KPP} type problems. {II}. {G}eneral
  domains.
\newblock {\em J. Amer. Math. Soc.}, 23(1):1--34, 2010.

\bibitem{BRR2}
H.~Berestycki, J.-M. Roquejoffre, and L.~Rossi.
\newblock Fisher-{KPP} propagation in the presence of a line: further effects.
\newblock {\em Nonlinearity}, 26(9):2623--2640, 2013.

\bibitem{BRR1}
H.~Berestycki, J.-M. Roquejoffre, and L.~Rossi.
\newblock The influence of a line with fast diffusion on {F}isher-{KPP}
  propagation.
\newblock {\em Journal of Mathematical Biology}, 66(4-5):743--766, 2013.

\bibitem{these_AC}
A.-C. Coulon-Chalmin.
\newblock {\em Propagation in reaction-diffusion equations with fractional
  diffusion}.
\newblock PhD thesis, Toulouse 3 - UPC Barcelona, 2014.

\bibitem{Dietrich1}
L.~Dietrich.
\newblock Existence of fronts in a reaction-diffusion system with a line of
  fast diffusion.
\newblock {\em Appl. Math. Res. Express}, to appear, 2015.

\bibitem{Dietrich2}
L.~Dietrich.
\newblock Velocity enhancement of reaction-diffusion fronts by a line of fast
  diffusion.
\newblock {\em Trans. Amer. Mat. Soc.}, to appear, 2015.

\bibitem{fisher}
R.~A. Fisher.
\newblock the advance of advantageous genes.
\newblock {\em ANN. Eugenics}, 7:335--369, 1937.

\bibitem{FG}
Ju. Gertner and M.~I. Fre{\u\i}dlin.
\newblock The propagation of concentration waves in periodic and random media.
\newblock {\em Dokl. Akad. Nauk SSSR}, 249(3):521--525, 1979.

\bibitem{GT}
D.~Gilbarg and N.~S. Trudinger.
\newblock {\em Elliptic Partial Differential Equations of Second Order}.
\newblock Springer-Verlag, 2001.

\bibitem{henry}
D.~Henry.
\newblock {\em Geometric Theory of Semilinear Parabolic Equation}.
\newblock Lecture Notes in Mathematics. Springer-Verlag, 1981.

\bibitem{KPP}
A.~Kolmogorov, I.~Petrovsky, and N.~Piskounov.
\newblock Etude de l'\'equation de la diffusion avec croissance de la
  quantit\'e de mati\`ere et son application \`a un probl\`eme biologique.
\newblock {\em Bull. Univ. Etat Moscou}, 1:1--26, 1937.

\bibitem{Lunardi}
A.~Lunardi.
\newblock {\em Analytic semigroups and optimal regularity in parabolic
  problems}.
\newblock Modern Birkh\"auser Classics. Birkh\"auser/Springer Basel AG, Basel,
  1995.

\bibitem{Pauthier}
A.~Pauthier.
\newblock The influence of a line with fast diffusion and nonlocal exchange
  terms on fisher-{KPP} propagation.
\newblock {\em Comm. Math. Sciences}, to appear, 2015.

\bibitem{W2002}
H.F. Weinberger.
\newblock On spreading speeds and traveling waves for growth and migration
  models in a periodic habitat.
\newblock {\em J. Math. Biol.}, 45(6):511--548, 2002.

\end{thebibliography}

\end{document}